\documentclass[a4paper]{article}

\usepackage{verbatim}

\usepackage{url}
\usepackage[numbers]{natbib}
\newcommand\citeayn[2][]{\citeauthor*{#2} (\citeyear{#2}) \cite[#1]{#2}}
\newcommand\citeny[2][]{\textnormal{\cite{#2}~(\citeyear{#2})}}

\usepackage{mathtools}
\usepackage{amsmath}
\usepackage{amssymb}
\usepackage{mathabx}
\usepackage{amsthm}
\usepackage{latexsym}
\usepackage{enumitem}
\usepackage{eurosym}
\usepackage{dsfont}
\usepackage{appendix}
\usepackage{color} 

\usepackage{frcursive}
\usepackage[utf8]{inputenc}
\usepackage{geometry}
\usepackage{multirow}
\usepackage{todonotes}
\usepackage{lmodern}
\usepackage{anyfontsize}
\usepackage{stmaryrd}
\usepackage{cases}
\usepackage{calc}
\usepackage{pdfpages}
\usepackage{indentfirst}

\usepackage[unicode,breaklinks]{hyperref}
\usepackage[nameinlink]{cleveref}
\newcommand*{\email}[1]{\href{mailto:#1}{\nolinkurl{#1}}} 

\Crefname{equation}{Condition}{Conditions}
\Crefname{problem}{Problem}{Problems}
\Crefname{assumption}{Assumption}{Assumptions}

\usepackage{mathrsfs}

\usepackage[english]{babel}
\usepackage[english=british]{csquotes}

\DeclareMathOperator*{\argmax}{arg\,max}
\DeclareMathOperator*{\argmin}{arg\,min}

\definecolor{red}{rgb}{0.7,0.15,0.15}
\definecolor{green}{rgb}{0,0.5,0}
\definecolor{blue}{rgb}{0,0,0.7}
\hypersetup{colorlinks, linkcolor={blue},citecolor={green}, urlcolor={red}}
			
\makeatletter \@addtoreset{equation}{section}

\usepackage{subcaption}
\usepackage[justification=centering]{caption}
\newtheorem{theorem}{Theorem}[section]
\newtheorem{assumption}{Assumption}

\newtheorem{lemma}[theorem]{Lemma}
\newtheorem{proposition}[theorem]{Proposition}

\newtheorem{definition}[theorem]{Definition}
\newtheorem{problem}{Problem}
\newtheorem{remark}[theorem]{Remark}

\makeatletter
\renewenvironment{proof}[1][\relax]{\par
  \pushQED{\qed}%
  \normalfont \topsep6\p@\@plus6\p@\relax
  \trivlist
  \item[\hskip\labelsep\itshape
    \ifx#1\relax \proofname\else\proofname{} of #1\fi\@addpunct{.}]\ignorespaces
}{%
  \popQED\endtrivlist\@endpefalse
}
\makeatother

\DeclareUnicodeCharacter{014D}{\=o}
\def \D{\mathbb{D}}
\def \E{\mathbb{E}}
\def \F{\mathbb{F}}

\def \H{\mathbb{H}}

\def \M{\mathbb{M}}
\def \N{\mathbb{N}}

\def \P{\mathbb{P}}

\def \R{\mathbb{R}}
\def \S{\mathbb{S}}

\def \U{\mathbb{U}}

\def\Cc{{\cal C}}

\def\Ec{{\cal E}}
\def\Fc{{\cal F}}

\def\Hc{{\cal H}}

\def\Kc{{\cal K}}

\def\Mc{{\cal M}}

\def\Oc{{\cal O}}
\def\Pc{{\cal P}}

\def\Rc{{\cal R}}
\def\Sc{{\cal S}}

\def\Uc{{\cal U}}
\def\Vc{{\cal V}}

\def\erm{\mathrm{e}}
\def\drm{\mathrm{d}}

\definecolor{bleudefrance}{rgb}{0.19, 0.55, 0.91}
\definecolor{darkspringgreen}{RGB}{60, 179, 113}
\newcommand{\AC}[1]{{\color{darkspringgreen} #1}}

\newcommand{\modifreview}[1]{#1}
\newcommand{\modif}[1]{#1}

\title{A new approach to principal--agent problems with volatility control}

\author{Alessandro {\sc Chiusolo}\thanks{Department of Operations Research and Financial Engineering (ORFE), Princeton University, NJ 08544. Research partially supported by the NSF grant DMS-2307736.} \and Emma {\sc Hubert}\footnotemark[1]}

\date{\today}

\begin{document}

\maketitle

\begin{abstract}
	The recent work by \citeayn{cvitanic2018dynamic} presents a general approach for continuous-time principal--agent problems, through dynamic programming and \emph{second-order} backward stochastic differential equations (BSDEs). In this paper, we provide an alternative formulation of the principal--agent problem, which can be solved simply by relying on the theory of BSDEs. This reformulation is strongly inspired by an important remark in \cite{cvitanic2018dynamic}, namely that if the principal observes the output process $X$ in continuous-time, she can compute its quadratic variation pathwise. While in \cite{cvitanic2018dynamic}, this information is used in the contract, our reformulation consists in assuming that the principal could directly control this process, in a `first-best' fashion. The resolution approach for this alternative problem actually follows the line of the so-called `Sannikov's trick' in the literature on continuous-time principal--agent problems, as originally introduced by \citeayn{sannikov2008continuous}. We then show that the solution to this `first-best' formulation is identical to the solution of the original problem. More precisely, using the contract form introduced in \cite{cvitanic2018dynamic} as \emph{penalisation contracts}, we highlight that this `first-best' scenario can be achieved even if the principal cannot directly control the quadratic variation. Nevertheless, we do not have to rely on the theory of 2BSDEs to prove that such contracts are optimal, as their optimality is ensured by showing that the `first-best' scenario is achieved. We believe that this more straightforward approach to solve continuous-time principal--agent problems with volatility control will facilitate the dissemination of these problems across many fields, and its extension to even more intricate problems.

	\bigskip

	
	\medskip
	
\end{abstract}

\section{Introduction}\label{sec:introduction}

Contract theory, and more precisely principal--agent problems, offer a relevant framework for the study of optimal incentives between agents, especially with information asymmetry. In standard principal--agent problems, a principal (she) is imperfectly informed about the actions of an agent (he). To incentivise the agent to act in her best interest, she can offer him a contract, usually represented by a terminal payment $\xi$ at a fixed finite time horizon $T > 0$. The main mathematical challenge of these problems is to determine an \emph{optimal} contract, in the sense that it maximises the utility of the principal, while that of the agent is held at least to a given level. The design of such a contract raises first of all an issue of information: the contract can only depend on what is known, or observable, by the principal. We focus in this paper on the moral hazard case, \textit{i.e.} when the \emph{effort} of the agent is hidden and therefore not contractible. 

\medskip

For a long time, these problems were only considered in discrete-time or static settings, in particular with the works of \citeauthor{mirrlees1976optimal} in \citeny{mirrlees1976optimal} and \citeny{mirrlees1999theory}, or the seminal papers of \citeayn{holmstrom1979moral} and \citeayn{grossman1983analysis}. However, discrete-time models are actually challenging to solve, especially because the existence of an optimal contract in these settings is more the exception than the rule, even in very simple and natural examples. The theory then regained momentum with the first continuous-time formulation of the problem by \citeayn{holmstrom1987aggregation}, and later by \citeayn{sannikov2008continuous}, using tools from stochastic control theory. Through these works, it is established that, if the agent only controls the drift of a stochastic process $X$, continuously observable by the principal, then the optimal contract $\xi$ should have the following form,
\begin{align}\label{eq:contract_drift}
	\xi = y_0 - \int_0^T \Hc (t, X_t, Z_t) \drm t + \int_0^T Z_t \cdot \drm X_t, 
\end{align}
where $\Hc$ here is the (simplified) agent's Hamiltonian for drift control. In the previous contract form, the constant $y_0 \in \R$ and the process $Z$ have to be optimally chosen by the principal, in order to maximise her utility and ensure the participation of the agent in the contractual agreement. More precisely, the choice of $y_0 \in \R$ is usually dictated by the agent's participation constraint, while the process $Z$ will incentivise the agent to perform appropriate effort on the drift of the output process. This form of contract naturally comes from the study of the agent's continuation utility, which, in the case of drift control only, can be related to the solution of a backward stochastic differential equation (BSDE for short). This formal link between the agent's continuation utility and the contract itself was introduced by \citeayn{sannikov2008continuous}, and thus usually referred to as `Sannikov's trick' in the literature on continuous-time principal--agent problems, but has recently been revisited in a more theoretical way by \citeayn{possamai2024there}. These problems have subsequently been extended in many directions, for example to consider multi-agent 
and/or multi-principal problems, 
a continuum of agents with mean-field interactions, 
more general output process $X$, more intricate information asymmetry between the two parties, to only cite a few. 

\medskip

In terms of theoretical developments, the most significant breakthrough is thanks to \citeayn{cvitanic2018dynamic}, who have developed a general and rigorous theory that allows to address a wide spectrum of principal--agent problems, especially when the agent can also control the volatility of the output process. As explained in \cite{cvitanic2018dynamic}, moral hazard problems in continuous-time in which the agent controls the volatility are notoriously harder to study. Indeed, in this case, his continuation utility cannot be related to the solution of a BSDE anymore, one has to rely now on the theory of backward stochastic differential equations of the \emph{second order} (2BSDE for short), introduced by \citeayn{soner2012wellposedness} and further developed by \citeayn{possamai2018stochastic}. One important observation made in \cite{cvitanic2018dynamic} which will also be at the heart of our approach is that, if the principal has access to the trajectory of the output process $X$, then she can deduce its quadratic variation process $\langle X \rangle$ in a pathwise manner, as demonstrated by \citeayn{bichteler1981stochastic}. This motivates the authors to introduce a new form of contracts, namely
\begin{align}\label{eq:contract_vol}
	\xi = y_0 - \int_0^T \Hc (t, X_t, Z_t, \Gamma_t) \drm t + \int_0^T Z_t \cdot \drm X_t + \dfrac12 \int_0^T {\rm Tr} \big[ \Gamma_t \drm \langle X \rangle_t \big].
\end{align}
These contracts are indexed on $X$ as before, but also on its quadratic variation $\langle X \rangle$ through an additional parameter $\Gamma$, which would incentivise the agent to perform appropriate efforts on the volatility. Note that the agent's Hamiltonian $\Hc$ here should be different from the one introduced in the previous contract form, precisely to take into account volatility control. When restricting the study to this form of contracts, it is straightforward to show, mostly using Ito's formula, that the agent's optimal effort will naturally be given by the maximiser of the Hamiltonian. One can then proceed by solving the principal's optimal control problem, namely by writing the Hamilton-Jacobi-Bellman (HJB for short) equation corresponding to her stochastic control problem, which becomes relatively standard under the restriction to the previous form of contracts. These first two steps are therefore quite straightforward using standard tools from stochastic control theory. However, the third and last step requires to prove that the restriction to such form of contracts is without loss of generality. While the approach provided in \cite{cvitanic2018dynamic} to prove this main result relies on the theory of 2BSDEs, the aim of this paper is precisely to present an alternative approach which avoids relying on this sophisticated theory. For the sake of comparison, we will study the same general model, whose original formulation will be recalled in \Cref{sec:model}.

\medskip

Despite the challenges raised when studying principal--agent problems with volatility control, especially as the optimality of revealing contracts relies on deep results from the recent theory of 2BSDEs, it has subsequently been extended and applied in many different situations. We can mention in a non-exhaustive way the applications related to finance by the same authors in \cite{cvitanic2017moral}, by \citeayn{cvitanic2018asset} or more recently by \citeayn{baldacci2022governmental}; the works by \citeayn{aid2022optimal}, \citeayn{elie2021meanfield} and \citeayn{aid2023principalagent} for applications related to the energy sector; as well as other various extensions, \textit{e.g.}, the works of \citeayn{mastrolia2018moral}, \citeayn{hernandezsantibanez2019contract}, \citeayn{hu2019continuoustime},  \citeayn{hubert2023continuoustime} or \citeayn{keppo2024dynamic}. Naturally, the study of applications fitting into the general framework proposed in \cite{cvitanic2018dynamic} can be done in a fairly straightforward way, directly using the results, as done in \cite{cvitanic2017moral,aid2022optimal}. Nevertheless, the recent research mentioned above has also attempted to generalise principal--agent problems, especially to consider many agents (see \cite{hubert2023continuoustime,aid2023principalagent}), or even a continuum (see \cite{elie2021meanfield}). 
If one follows the approach developed in \cite{cvitanic2018dynamic}, these extensions require to introduce the so-called \emph{multidimensional} or \emph{mean-field} 2BSDEs, and study, or assume, the well-posedness of such systems. However, while the theory on (one-dimensional) 2BSDEs is sufficiently developed to address relatively general principal--agent problems, there is currently no definition nor existence results on multidimensional/mean-field 2BSDEs. This further motivates the presentation in this paper of an alternative approach. We believe that this novel approach, which avoids relying on the theory of 2BSDEs, will greatly facilitate the study of relevant generalised principal--agent problems, especially with many agents, or non-continuous output processes.

\medskip

Before introducing the approach developed in this paper, we wish to provide more intuition on the problem under consideration via a short detour through the so-called \emph{first-best} case. When studying principal--agent problems with moral hazard, it is quite common to compare the results with this benchmark scenario, in which the principal can directly optimise the agent's effort. This problem is usually easier to solve, as one only has to consider the principal's stochastic control problem under the agent's participation constraint. It may happen in the moral-hazard case, sometimes also referred to as \emph{second-best}, that the principal can actually deduce, through her observation, the effective efforts of the agent. In this scenario, the two problems, first-best and second-best, are then equivalent if the principal can design a contract which strongly penalises the agent whenever he deviates from the optimal effort recommended by the principal. Typically, the principal will offer to the agent a minimal\footnote{Minimal in the sense that it would just be sufficient to ensure the participation of agent.} compensation if he follows the recommendation, and threaten him with a high penalty if he deviates from it. These \emph{forcing contracts} can usually be implemented if no strong assumptions are made on the admissible contract's form.

\medskip

As highlighted before, the approach developed by \citeayn{cvitanic2018dynamic} relies in particular on the observation that, if the principal have access to the trajectory of the output process $X$, she can actually compute its quadratic variation $\langle X \rangle$ pathwise, and thus design contracts of the form \eqref{eq:contract_CPT}, indexed on $X$ and $\langle X \rangle$ through sensitivity parameters $Z$ and $\Gamma$ to be optimised. The main idea behind our reformulation of the original principal--agent problem with volatility control is that, since the principal can deduce the quadratic variation $\langle X \rangle$, she could \emph{probably} force the agent to choose his efforts in order to achieve an \emph{optimal} quadratic variation. In other words, the idea here is to study an alternative `first-best' formulation of the original problem, in which the principal enforces a contract $\xi$ as well as a quadratic variation process. The agent should then optimise his effort in order to maximise his expected utility given such contract $\xi$ as before, but now the admissible set of efforts is restricted to those achieving the quadratic variation imposed by the principal.

\medskip

When considering the agent's problem in this `first-best' alternative formulation, described with more details in \Cref{ss:fb_formulation}, one can remark that, since the quadratic variation is \textit{a priori} fixed by the principal, the agent's weak formulation boils down to choosing equivalent probability measures by means of Girsanov theorem. For a fixed quadratic variation process and any contract $\xi$, the agent's continuation utility can therefore be represented by means of a BSDE, instead of a 2BSDE in the original formulation. More informally, one can actually apply the so-called `Sannikov's trick' to obtain a relevant representation of the contract $\xi$, which is then similar to the optimal form of contracts in principal--agent problems with drift control only, namely \eqref{eq:contract_drift}. Following this straightforward approach, we establish in \Cref{ss:approach} that the principal's problem boils down to a more standard stochastic control problem, in the sense that she should maximise her expected utility by optimally choosing the contract's parameters, namely a constant $y_0 \in \R$ and an appropriate process $Z$, as well as the quadratic variation process. 

\medskip

However, as mentioned before, the moral-hazard case is equivalent to its first-best counterpart if the principal can deduce the agent's control from her observations \emph{and} design a forcing contract. Indeed, even if the principal observes the agent's control, she may not have enough \emph{bargaining power} to actually force the agent to perform the optimal efforts. Similarly, in our problem, even if the principal observes the quadratic variation of the output process, she may not be able to actually force the agent to follow the recommended quadratic variation. Therefore, to ensure the equivalence between the original problem and its `first-best' reformulation, we need to introduce appropriate forcing contracts, which allow the principal to achieve her `first-best' utility even if she cannot directly control the quadratic variation. It is already quite clear that the contracts introduced in \cite{cvitanic2018dynamic} will perfectly serve this purpose. Indeed, as these contracts are optimal when the principal observes the quadratic variation, they should allow her to achieve the expected utility she would have if she could directly control the quadratic variation process. 
We thus use in \Cref{ss:penalisation_contract} the contract form \eqref{eq:contract_vol} to conclude that the principal's value in the reformulated problem coincides with the one derived in \cite{cvitanic2018dynamic}. In other words, from the principal's point of view, optimising over $\Gamma$ in the contract form \eqref{eq:contract_vol} or directly optimising over admissible quadratic variation lead to the exact same results. In particular, this result confirms that the contracts introduced in \cite{cvitanic2018dynamic} are also optimal when the agent's efforts on the volatility can actually be determined by the principal through her observation of the quadratic variation. \modif{Interestingly, our approach further reveals the importance of an additional assumption under which the contract form \eqref{eq:contract_vol} is indeed optimal and the above equivalence holds. Specifically, if \Cref{ass:duality} is not satisfied, the method of~\cite{cvitanic2018dynamic} may fail to yield optimal contracts, as explained in \Cref{rk:duality} and illustrated through a counter-example in \Cref{ss:counter-example}.}

\medskip

To summarise, this paper provides an alternative yet more accessible approach to solve principal--agent problems with volatility control, in particular to prove that the contract form introduced in \cite{cvitanic2018dynamic} is without loss of generality from the principal's point of view. More precisely, we first solve a `first-best' reformulation of the original problem by means of the theory of BSDEs, and then show that the contract form introduced in \cite{cvitanic2018dynamic} allows to achieve this first-best scenario, and is therefore naturally optimal. We believe that this \emph{simpler} approach will enable better dissemination of these problems in various fields across mathematical finance, economics, management sciences, operation research; fostering their use in many relevant applications. Furthermore, as this new approach only relies on the theory of BSDEs, we believe that it will facilitate the study of relevant extensions of principal--agent problems, for example when considering such a Stackelberg game with many agents, as well as its mean-field limit, when the principal can contract with a continuum of agents. \modifreview{We nevertheless emphasise that our approach relies on \Cref{ass:weak_uniqueness}, which is not assumed in \cite{cvitanic2018dynamic}. More precisely, to apply the BSDE theory without introducing an orthogonal martingale component, we require a martingale representation theorem to hold; \Cref{ass:weak_uniqueness} provides a sufficient condition for this. We refer the reader to \Cref{rk:MRP1,rk:MRP2} for a more detailed discussion. While this assumption appears necessary for our current arguments, we believe it constitutes a relatively mild condition, especially given the significant simplification afforded by our approach. Extending the method beyond the scope of \Cref{ass:weak_uniqueness} remains an open question, but we believe there is good reason to expect that such generalisations are possible and would preserve the validity of our main results.}


\medskip

The structure of the paper is as follows. We first introduce in \Cref{sec:model} the continuous-time principal--agent problem considered by \citeayn{cvitanic2018dynamic}. We also briefly describe the main steps of the approach they developed to solve the problem, highlighting the main challenges posed by this approach. \Cref{sec:new_approach} details the main results of this paper: we first introduce in \Cref{ss:fb_formulation} the alternative `first-best' reformulation of the original principal--agent problem, we then characterise in \Cref{ss:approach} the optimal form of contracts, relying only on the theory of BSDEs, and we finally prove in \Cref{ss:penalisation_contract} that the solution of the reformulated problem coincides with the solution of the original problem. To further highlight the link between the two formulations, we provide in \Cref{sec:examples} some illustrative examples in which the solution can be explicitly characterised. Finally, \Cref{sec:conclusion} concludes by providing further related research directions.

\paragraph*{Notations.} Let $\mathbb{N}^\star:=\mathbb{N}\setminus\{0\}$ be the set of positive integers, and recall that $T > 0$ denotes some maturity fixed in the contract. For any $(\ell,c) \in \N^\star \times \N^\star$, $\M^{\ell,c}$ will denote the space of $\ell\times c$ matrices with real entries. When $\ell=c$, we let $\M^{\ell}:=\M^{\ell,\ell}$, and denote by $\S_+^{\ell}$ the set of $\ell \times \ell$ non-negative definite (or positive semi-definite) symmetric matrices with real entries. For any $d \in \N^\star$, let $\Cc ( [0,T], \R^d)$ denote the set of continuous functions from $[0,T]$ to $\R^d$. The set $\Cc( [0,T], \R^d)$ will always be endowed with the topology associated to the uniform convergence on the compact $[0,T]$. For a probability space of the form $\Omega := \Cc ( [0,T], \R^d)$ and an associated filtration $\F$, we will have to consider processes $\psi : [0,T] \times \Cc ( [0,T], \R^k) \longrightarrow E$, taking values in some Polish space $E$, which are $\F$-optional, \textit{i.e.} $\Oc(\F)$-measurable where $\Oc(\F)$ is the so-called optional $\sigma$-field generated by $\F$-adapted right-continuous processes. In particular, such a process $\psi$ is non-anticipative in the sense that $\psi (t,x) = \psi(t, x_{\cdot \wedge t})$, for all $t\in[0,T]$ and $x \in \Cc ( [0,T], \R^k)$. Furthermore, for any filtration $\F := (\Fc_t)_{t \in [0,T]}$ and probability measure $\P$, we will denote by:
\begin{itemize}[label=\raisebox{0.25ex}{\tiny$\bullet$}]
\itemsep0em
    \item $\F^\P := (\Fc_t^\P)_{t \in [0,T]}$ the usual $\P$-completion of $\F$;
    \item $\F^+ := (\Fc_t^+)_{t \in [0,T]}$ the right limit of $\F$;
    \item $\F^{\P^+} := (\Fc_t^{\P^+})_{t \in [0,T]}$ the augmentation of $\F$ under $\P$.
\end{itemize}
Finally, for technical reasons, we work under the classical ZFC set-theoretic axiom. This is a standard axiom widely used in the classical stochastic analysis theory (see \citeayn[Chapter 0.9]{dellacherie1978probabilities}). We also rely on the continuum hypothesis, which is needed to apply the aggregation result of \citeayn{nutz2012pathwise}.

%

\section{General model}\label{sec:model}

In order to highlight that our alternative approach to principal--agent problems with volatility control allows to obtain the same results as in \cite{cvitanic2018dynamic}, we focus here on the exact same model. We believe that this model is sufficiently general to convince the reader of the relevance of our novel approach, and that it could then be extended in a (relatively) straightforward way to solve more complex problems.

\medskip

We fix throughout this paper a finite time horizon $T > 0$, and consider the canonical space $\Omega=\Cc([0,T];\R^d)$, $d \in \N^\star$. The output process $X$ controlled by the agent is originally defined as the canonical process on $\Omega$, \textit{i.e.}
\begin{align*}
	X_t(x) = x(t), \quad \text{for all } \; x \in \Omega, \; t\in[0,T].
\end{align*}
We denote by $\F=(\Fc_t)_{t\geq 0}$ its canonical filtration, defined for all $t\in[0,T]$ by $\Fc_t = \sigma(X_s, s \in [0,t])$. Let $\Mc$ be the set of all probability measures on $(\Omega,\Fc_T)$, and define $\Mc_{\rm loc}$ the subset of $\Mc$ such that, for each $\P \in \Mc_{\rm loc}$, $X$ is a $(\P,\F)$--local martingale whose quadratic variation, denoted $\langle X\rangle$, is absolutely continuous in time with respect to the Lebesgue measure. Using the result by \citeayn{karandikar1995pathwise}, the quadratic variation process $\langle X\rangle$ can be defined pathwise, in the sense that there exists an $\F$--progressively measurable process which coincides with the quadratic variation of $X$, $\mathbb P$--a.s. for any $\P\in \Mc_{\rm loc}$. Moreover, we can introduce the corresponding $\F$--progressively measurable density process $\widehat \sigma^2$, defined by
\begin{align}\label{eq:qvar_pathwise}
	\widehat \sigma^2_t := \limsup_{\varepsilon \searrow 0}  \dfrac{\langle X \rangle_t - \langle X \rangle_{t-\varepsilon}}{\varepsilon}, \quad t \in [0,T].
\end{align}
For all $t \in [0,T]$, $\widehat \sigma^2_t \in \S_+^{d}$, recalling that $\S_+^{d}$ denotes the set of $d \times d$ positive semi-definite symmetric matrices with real entries. We can thus define the corresponding square root process, simply denoted $\widehat \sigma$, in the usual way.

\subsection{Controlled output process}

To represent the agent's controls on the drift and the volatility of the output process $X$, we let $U$ be a finite-dimensional Euclidean set and consider the following coefficients:
\begin{align*}
	\lambda: [0,T] \times \Omega \times U \longrightarrow \R^n, \quad
	\text{and} \quad
	\sigma: [0,T] \times \Omega \times U \longrightarrow \M^{d,n}.
\end{align*}
As in \cite{cvitanic2018dynamic}, we assume that both coefficients are bounded, and that the mappings $\lambda(\cdot, u)$ and $\sigma(\cdot, u)$ are $\F$-optional for any $u \in U$. In particular, recall that such mappings are non-anticipative, in the sense that $\lambda(t,x,u) = \lambda(t, x_{\cdot \wedge t},u)$ and $\sigma(t,x,u)= \sigma(t, x_{\cdot \wedge t},u)$, for all $t\in[0,T]$ and $x \in \Omega:=\Cc ( [0,T], \R^d)$. We will rigorously define later the set $\Uc$ of admissible controls, namely appropriate processes taking values in $U$, and consider, for all $\nu \in \Uc$, the following stochastic differential equation (SDE for short), driven by an $n$-dimensional Brownian motion $W$, for $n \in \N^\star$,
\begin{align*}
	\drm X_t = \sigma (t, X,\nu_t) \big( \lambda(t, X,\nu_t) \mathrm{d}t +  \mathrm{d}W_t \big), \; t\in[0,T].
\end{align*}

To rigorously introduce the weak formulation for the agent's problem, one may need to enlarge the original canonical space by considering
\[
\widebar \Omega := \Omega \times \Cc ([0,T], \R^{n}) \times \U,
\]
where $\mathbb U$ is the collection of all finite and positive Borel measures on $[0,T] \times U$, whose projection on $[0,T]$ is the Lebesgue measure. In other words, every $q \in \mathbb \U$ can be disintegrated as $q(\mathrm{d}t,\mathrm{d} u)=q_t(\mathrm{d} u)\mathrm{d}t$, for an appropriate Borel measurable kernel $q_t$, $t \in [0,T]$. The weak formulation requires to consider a subset of $\U$, namely the set $\U_0$ of all $q \in \U$ such that the kernel $q_t$ is of the form $\delta_{\phi_t}(\mathrm{d} u)$ for some Borel function $\phi$, where as usual, $\delta_{\phi_t}$ is the Dirac mass at $\phi_t$, $t \in [0,T]$.
This space is supporting a canonical process $(X, W, \Pi)$ defined by:
\begin{align*}
	X_t(x,\omega,q) &:=x(t), \; W_t(x,\omega,q):= \omega(t), \; \Pi(x,\omega,q):=q, \; (t, x, \omega, q) \in [0,T] \times \widebar \Omega.
\end{align*}
Then, the canonical filtration $\widebar \F:=(\widebar \Fc_t)_{t\in[0,T]}$ is defined as the following $\sigma$-algebra,
\begin{align*}
	\widebar \Fc_t:=\sigma\Big( \big( X_s, W_s, \Delta_s(\varphi) \big) \text{ s.t. }  (s,\varphi)\in[0,t]\times \Cc_b \big( [0,T]\times U, \R \big) \Big),\; t\in [0,T],
\end{align*}
where $\Cc_b([0,T]\times U,\R)$ is the set of all bounded continuous functions from $[0,T]\times U$ to $\R$, and
\begin{align*}
	\Delta_s(\varphi):=\int_0^s\int_{U} \varphi(r,u) \Pi (\mathrm{d}r,\mathrm{d} u), \text{ for any } \; (s,\varphi)\in[0,T]\times \Cc_b([0,T]\times U,\R).
\end{align*}
Finally, let $\Cc^2_b(\R^k,\R)$ for any $k \in \N^\star$ denotes the set of bounded twice continuously differentiable functions from $\R^k$ to $\R$, whose first and second derivatives are also bounded. For any $(t,\psi)\in[0,T]\times \Cc^2_b(\R^{d+n}, \R)$, we set
\begin{align}\label{eq:Mphi}
	M_t(\psi):= \psi (X_t, W_t) - \int_0^t \int_{U} \bigg( &\ \widebar \Lambda (s, X, u) \cdot \nabla \psi (X_s, W_s)
	+ \frac12  {\rm Tr} \Big[ \nabla^2 \psi (X_s, W_s) \big[ \widebar \Sigma \widebar \Sigma^\top \big] (s, X, u) \Big] \bigg) \Pi (\mathrm{d}s, \mathrm{d} u),
\end{align}
where $\nabla \psi$ and $\nabla^2 \psi$ denote the gradient vector and the Hessian matrix of $\psi$, respectively, and where $\widebar \Lambda$ and $\widebar \Sigma$ are respectively the drift vector and the diffusion matrix of the $(d+n)$-dimensional vector process $(X, W)^\top$, \textit{i.e.}
\renewcommand*{\arraystretch}{1}
\begin{align*}
	\widebar \Lambda(s, x, u) := 
	\begin{pmatrix}
		[\sigma \lambda] (s, x , u) \\
		\mathbf{0}_{n}
	\end{pmatrix},
	\; 
	\widebar \Sigma(s,x, u) :=
	\begin{pmatrix}
		\mathbf{0}_{d, d} & \sigma (s,x,u)  \\
		\mathbf{0}_{n, d} & \mathrm{I}_{n} \\
	\end{pmatrix}, \; s \in [0,T], \; x \in \Omega, \; u \in U.
\end{align*}

Let $\widebar \Mc$ be the set of all probability measures on $(\widebar \Omega,\widebar \Fc_T)$. We fix some initial conditions, namely $x_0 \in \R^{d}$ representing the initial value of $X$, and define the following subset of $\widebar \Mc$.
\begin{definition}\label{def:weak_formulation}
\itemsep0em
	The subset $\widebar \Pc \subset \widebar \Mc$ is composed of all $\widebar \P$ such that
	\begin{enumerate}[label=$(\roman*)$]
		\item $M(\psi)$ is a $(\widebar \P, \widebar \F)$-local martingale on $[0,T]$ for all $\psi \in \Cc^2_b(\R^{d+n},\R);$
		\item there exists some $w_0 \in \R^{n}$ such that $\widebar \P [(X_0,W_0) = (x_0,w_0)]=1;$
		\item $\widebar \P\big[\Pi \in \U_0]=1$.
	\end{enumerate}
\end{definition}
Using the classical result by \citeayn[Theorem 4.5.2]{stroock1997multidimensional} in our already enlarged canonical space, one can rigorously \modifreview{define} the required dynamics for $X$ under $\P \in \widebar \Pc$. In particular, we have that, for all $\P \in \widebar \Pc$, there exists some $\widebar \F$-predictable process $\nu := (\nu_t)_{t \in [0,T]}$, taking values in $U$, such that $\Pi(\mathrm{d}s,\mathrm{d} u) = \delta_{\nu_s}(\mathrm{d} u)\mathrm{d}s $ $\P$--{\rm a.s.}, and we obtain the desired representation for $X$, namely
\begin{align}\label{eq:SDE_drift}
	X_t = x_0 + \int_0^t \sigma (s, X,\nu_s) \big( \lambda(s, X,\nu_s) \mathrm{d}s +  \mathrm{d}W_s \big), \; t\in[0,T],\; \P \textnormal{--a.s.}
\end{align}
With this in mind, for any $\P \in \widebar \Pc$, we denote by $\widebar \Uc(\P)$ the set of $\widebar \F$-predictable and $U$-valued processes $\nu$ such that \eqref{eq:SDE_drift} holds, and consider
\begin{align*}
    \widebar \Uc := \bigcup_{\P \in \widebar \Pc} \widebar \Uc(\P),
\end{align*}
as well as for any $\nu \in \widebar \Uc$, the set of corresponding probability measures $\widebar \Pc(\nu) := \{ \P \in \widebar \Pc : \nu \in \widebar \Uc(\P) \}$. Finally, to avoid technical considerations (see \Cref{rk:MRP1} below), we make the following standing assumption.

\modifreview{
\begin{assumption}\label{ass:weak_uniqueness}
Each agent's admissible controls induce a unique weak solution to {\rm SDE} \eqref{eq:SDE_drift}, \textit{i.e.} $\nu \in \Uc$ with
\begin{align}\label{def:admissible_effort}
    \Uc := \big\{ \nu \in \widebar \Uc : \widebar \Pc(\nu) \; \text{is a singleton} \; \{ \P^\nu \} \big\}.
\end{align}
\end{assumption}}
We thus have a one-to-one correspondence between any admissible control $\nu \in \Uc$ and a probability measure $\P \in \Pc$, where $\Pc$ is defined accordingly:
\begin{align*}
    \Pc := \bigcup_{\nu \in \Uc} \widebar \Pc(\nu).
\end{align*}
In the following, we will alternatively refer to the unique probability measure in $\Pc$ associated with an admissible control $\nu \in \Uc$ as $\P^\nu$, or the unique admissible control in $\Uc$ associated with a probability measure $\P \in \Pc$ as $\nu^\P$. 

\begin{remark}\label{rk:MRP1}
    {\rm \Cref{ass:weak_uniqueness}} is not enforced in {\rm \cite{cvitanic2018dynamic}}, but is considered here to avoid technical considerations regarding the martingale representation property. We refer to {\rm \Cref{rk:MRP2}} for further discussions on this point, but wish to justify here that this assumption is not far-fetched. Indeed, it is relatively common in the study of principal--agent problems to at least assume that the martingale representation is valid {\rm(}see {\rm \citeayn[Assumption 2.2. and Remark 2.3]{hernandez-santibanez2024principalmultiagents}} for example{\rm)}, and weak uniqueness for {\rm SDE}~\eqref{eq:SDE_drift} is a sufficient and more explicit condition for such representation to hold. Furthermore, as one of the main objectives of principal--agent problems is to determine `revealing' contracts, \textit{i.e.} contracts that reveal the agent's optimal efforts, it is actually questionable to consider efforts that would induce several weak solutions for {\rm SDE}~\eqref{eq:SDE_drift}, especially in view of applications.
\end{remark}

\subsection{The original principal--agent problem}

As usual in principal--agent problems, the principal first chooses a contract $\xi$, which represents the compensation paid to the agent at the final time horizon $T$. Naturally, this contract can only be indexed on what is observable by the principal, and should incentivise the agent to appropriately choose a control $\nu \in \Uc$, or equivalently a probability $\P \in \Pc$. Here, the principal only observes the output process $X$, and therefore any admissible contract $\xi$ should be $\Fc_T$-measurable, recalling that $\F := (\Fc_t)_{t \in [0,T]}$ represents the canonical filtration of $X$.

\paragraph*{The agent's problem.} Given a compensation $\xi$, the optimisation problem faced by the agent is defined by
\begin{align}\label{eq:pb_agent}
	V_{\rm A} (\xi) := \sup_{\P \in \Pc} J_{\rm A} (\xi,\P), \; \text{ with } \; 
    J_{\rm A} (\xi, \P ) := \E^\P \bigg[ \Kc_{\rm A}^\P(T) \, \xi - \int_0^T \Kc_{\rm A}^\P(s) c_{\rm A} \big(s, X, \nu_s^{\P} \big) \mathrm{d}s \bigg], \quad \P \in \Pc,
\end{align}
where
\begin{enumerate}[label=$(\roman*)$]
\itemsep0em
	\item the measurable cost function $c_{\rm A} : [0,T] \times \Cc([0,T], \modifreview{\R^d}) \times U \longrightarrow \R$ is such that for any $u \in U$, the map $(t,x) \longmapsto c_{\rm A}(t,x,u)$ is $\F$-optional, and satisfies the following integrability condition:
	\begin{align}\label{eq:integrability_cost}
		\sup_{\P\in \Pc} \E^\P \bigg[ \int_0^T \big| c_{\rm A} \big(t, X, \nu^{\P}_t \big) \big|^p \drm t \bigg] < + \infty, \; \text{for some } p >1;
	\end{align}
	\item the discount factor $\Kc_{\rm A}^\P$ is defined for all $\P \in \Pc$ by
	\[
	\Kc_{\rm A}^\P(t) :=\exp\bigg(-\int_0^t k_{\rm A} \big(s, X, \nu_s^\P \big) \drm s \bigg), \quad t \in [0,T],
	\]
	where $k_{\rm A} : [0,T] \times \Cc([0,T], \modifreview{\R^d}) \times U \longrightarrow \R$ is assumed to be measurable, bounded, and such that for any $u \in U$, the map $(t,x)\longmapsto k_{\rm A}(t,x,u)$ is $\F$-optional.
\end{enumerate}
To ensure that the agent's problem \eqref{eq:pb_agent} is well-defined, we require minimal integrability on the contracts as in \cite[Equation  (2.5)]{cvitanic2018dynamic}, namely
\begin{align}\label{eq:integrability_contract_agent}
	\sup_{\P\in \Pc} \E^\P \big[ |\xi |^p \big] < + \infty, \; \text{ for some } \; p > 1.
	\tag{${\rm I}^p_{\rm A}$}
\end{align}
In addition, we assume as usual that the agent has a reservation utility level $R_{\rm A} \in \R$, which conditions his acceptance of the contract. We thus denote by $\Xi$ the set of admissible contracts:
\begin{align}\label{def:contract_set}
	\Xi := \big\{\Fc_T\text{-measurable random variable $\xi$, satisfying \eqref{eq:integrability_contract_agent} and} \; V_{\rm A} (\xi) \geq R_{\rm A} \big\}.
\end{align}
Finally, we denote by $\Pc^\star(\xi)$ the set of optimal responses of the agent to a contract $\xi \in \Xi$, \textit{i.e.} 
\begin{align}\label{def:optimal_response}
    \Pc^\star(\xi) := \big\{ \P^\star \in \Pc \; \text{such that} \; V_{\rm A} (\xi) = J_{\rm A} (\xi,\P^\star) \big\}.
\end{align}

\begin{remark}
    As mentioned in {\rm \cite[Remark 2.1]{cvitanic2018dynamic}}, the model and the approach can be extended in a straightforward way to account for risk-aversion in the agent’s utility function. More precisely, in the agent's expected utility \eqref{eq:pb_agent}, $\xi$ can be replaced by $U_{\rm A}(\xi)$ for an appropriate utility function $U_{\rm A}$, modulo the modification of the integrability condition \eqref{eq:integrability_contract_agent}. As we choose to follow the model considered in {\rm \cite{cvitanic2018dynamic}}, we do not consider any utility functions here, but we will nevertheless informally address a case of {\rm CARA} utility functions through the examples in {\rm \Cref{sec:examples}}.
\end{remark}

\paragraph*{The principal's problem.} The principal's optimisation problem is then defined as follows,
\begin{align}\label{eq:pb_principal}
	V_{\rm P} := \sup_{\xi \in \Xi} \sup_{\P^\star \in \Pc^\star(\xi)} J_{\rm P} (\xi,\P^\star), \; \text{ with } \; 
    J_{\rm P} (\xi, \P ) := \E^\P \Big[ \Kc_{\rm P}(T) \, U_{\rm P} \big( \ell(X)-\xi \big) \Big], \quad (\xi,\P) \in \Xi \times \Pc,
\end{align}
where
\begin{enumerate}[label=$(\roman*)$]
\itemsep0em
	\item $U_{\rm P} : \R \longrightarrow \R$ is a non-decreasing and concave utility function;
	\item $\ell : \Omega \longrightarrow \R$ is a liquidation function with linear growth;
	\item the discount factor $\Kc_{\rm P}$ is defined as follows, with $k_{\rm P}: [0,T] \times \Omega \longrightarrow \R$ measurable, bounded and $\F$-optional,
	\[
	\Kc_{\rm P} (t) := \exp\bigg(-\int_0^t k_{\rm P} (s, X) \drm s \bigg), \quad 0\leq t\leq s\leq T.
	\]
\end{enumerate}
Through \eqref{eq:pb_principal}, we implicitly assume as in \cite{cvitanic2018dynamic} that, if the agent's optimal response to a given contract $\xi$ is not unique, then the principal is able to choose the most appropriate one. On the contrary, if a contract $\xi \in \Xi$ is such that $\Pc^\star(\xi) = \varnothing$, then by convention of the supremum, the corresponding expected utility for the principal would be $- \infty$. We can thus assume without loss of generality that any relevant contract $\xi \in \Xi$ should ensure that there exists at least an optimal response $\P^\star \in \Pc^\star(\xi)$, \textit{i.e.} $\Pc^\star(\xi) \neq \varnothing$.

\medskip

For future reference, we summarise this original problem below.
\begin{problem}[Original principal--agent problem]\label{pb:PA_original}
Consider the non-standard stochastic control problem \eqref{eq:pb_principal},
\begin{align*}
	V_{\rm P} := \sup_{\xi \in \Xi} \sup_{\P^\star \in \Pc^\star(\xi)} J_{\rm P} (\xi, \P^\star), \; \text{ with } \; \Pc^\star(\xi) \neq \varnothing \; \text{ for } \; \xi \in \Xi.
\end{align*}
\end{problem}
Usually, \Cref{pb:PA_original} cannot be solved directly: the standard approach is to determine an appropriate form of contracts which simplifies both the agent's and the principal's problems, and such that the restriction to this form of contracts is without loss of generality.

\medskip

Note that the previous setting and assumptions are relatively standard in continuous-time principal--agent problems, and in particular coincide with the framework described in \cite{cvitanic2018dynamic}. Indeed, as the goal of this paper is to provide a new approach to solve such principal--agent problems, which in particular requires to show that this approach allows to achieve the same results obtained by \citeayn{cvitanic2018dynamic}, the decision was made to closely follow their original framework. It is nevertheless possible that some assumptions could be relaxed, or at least modified to consider alternative frameworks. For example, in the case of a principal--agent problem with non-continuous output processes (and many agents), one can refer to the recent paper by \citeayn{hernandez-santibanez2024principalmultiagents} for more general assumptions. It should be noted, however, that the latter focuses on control processes that can be represented through equivalent changes of measures, thus \textit{a priori} excluding volatility control. 

\subsection{The existing approach, via 2BSDEs}\label{ss:summary_CPT}

In this section, we describe briefly and rather informally the general approach developed by \citeayn{cvitanic2018dynamic}, as this will help in comparing it with our new approach, detailed later in \Cref{sec:new_approach}. 

\paragraph*{Revealing contracts.} The general approach developed in \cite{cvitanic2018dynamic} involves first identifying a class of contracts, offered by the principal, that are \emph{revealing}, in the sense that the agent's optimal response to the contract can easily be computed. With this in mind, we first define the appropriate agent's Hamiltonian,
\begin{align}\label{def_hamiltonian_CPT}
	\Hc_{\rm A} (t, x, y, z, \gamma) &:= \sup_{u \in U} \; h_{\rm A} (t, x, y, z, \gamma, u), \\ \text{ with } \; 
	h_{\rm A} (t, x, y, z, \gamma, u) &:= [\sigma \lambda ] (t, x, u) \cdot z + \dfrac12 {\rm Tr} \big[ \gamma [ \sigma \sigma^\top](t, x, u) \big] - c_{\rm A}(t, x, u) - k_{\rm A}(t,x,u) y, \quad u \in U, \nonumber
\end{align} 
for any $(t,x,y,z,\gamma) \in [0,T] \times \Omega \times \R \times \R^d \times \M^d$.
We then introduce the relevant form of contracts, namely $\xi = Y_T^{y_0,Z,\Gamma}$, where the process $Y^{y_0,Z,\Gamma}$ is defined as the solution to the following \modifreview{SDE},
\begin{align}\label{eq:contract_CPT}
	Y_t^{y_0,Z,\Gamma} = y_0 - \int_0^t \Hc_{\rm A} \big(s, X, Y^{y_0,Z,\Gamma}_s, Z_s, \Gamma_s \big) \drm s + \int_0^t Z_s \cdot \drm X_s + \dfrac12 \int_0^t {\rm Tr} \big[ \Gamma_s \drm \langle X \rangle_s \big], \quad \P\textnormal{--a.s.}, \; \text{for all } \P \in \Pc,
\end{align}
for a constant $y_0 \in \R$ and a pair $(Z,\Gamma)$ of processes satisfying appropriate integrability conditions. More precisely, for any $\F$-predictable $\R^d$-valued process $Z$ and any $\F$-optional $\R$-valued process $Y$, we consider the following norms,
\begin{align}\label{eq:norm}
	\| Z \|^p_{\H^p} := \sup_{\P \in \Pc} \E^\P \Bigg[ \bigg( \int_0^T  Z_t^\top \widehat \sigma^2_t Z_t \drm t \bigg)^{p/2} \Bigg],
	\; \text{ and } \; 
	\| Y \|^p_{\D^p} := \sup_{\P \in \Pc} \E^\P \bigg[ \sup_{t \in [0,T]} \big| Y_t \big|^p \bigg].
\end{align}
We can then define the set $\Vc$ of appropriate processes $(Z,\Gamma)$ as in \cite[Definition 3.2]{cvitanic2018dynamic}, recalled below.
\begin{definition}\label{def:contrat_vol}
	Let $y_0 \in \R$. For any $\F$-predictable processes $(Z,\Gamma)$ taking values in $\R^d \times \M^d$, define the process $Y^{y_0,Z,\Gamma}$ using \eqref{eq:contract_CPT}. We will denote $(Z,\Gamma) \in \Vc$ if moreover
	\begin{enumerate}[label=$(\roman*)$]
		\item $\| Z \|^p_{\H^p} < \infty$ and $\| Y^{y_0,Z,\Gamma} \|^p_{\D^p} < \infty$ for some $p > 1$;
		\item there exists $\P \in \Pc$ such that
		\begin{align*}
			\Hc_{\rm A} \big(t, X, Y_t^{y_0,Z,\Gamma}, Z_t, \Gamma_t \big) = h_{\rm A} \big(t, X, Y_t^{y_0,Z,\Gamma}, Z_t, \Gamma_t, \nu^\P_t \big), \quad \drm t \otimes \P\text{--a.e. on } [0,T] \times \Omega.
		\end{align*}
	\end{enumerate}
\end{definition}
Then, the relevant contracts $\xi \in \Xi$ are those that admit the representation $\xi = Y_T^{y_0,Z,\Gamma}$, where the pair $(Z,\Gamma) \in \Vc$ as well as the constant $y_0 \in \R$ in \eqref{eq:contract_CPT} will be optimally chosen by the principal.

\paragraph*{Agent's optimal response.} Contracts of the form \eqref{eq:contract_CPT} are called \emph{revealing}, as it is straightforward to prove that any agent's best response to such contracts corresponds to a maximiser of his Hamiltonian. Indeed, relying on Ito's formula, it is proved in \cite[Proposition 3.3]{cvitanic2018dynamic} that, given a contract of the form $\xi = Y_T^{y_0,Z,\Gamma}$, $(y_0,Z,\Gamma) \in \R \times \Vc$, any optimal response $\P^\star \in \Pc^\star(\xi)$ is such that 
\begin{align*}
	\Hc_{\rm A} \big(t, X, Y_t^{y_0,Z,\Gamma}, Z_t, \Gamma_t \big) = h_{\rm A} \big(t, X, Y_t^{y_0,Z,\Gamma}, Z_t, \Gamma_t, \nu^{\P^\star}_t \big), \quad \drm t \otimes \P^\star\textnormal{--a.e. on } [0,T] \times \Omega. 
\end{align*}
For later use, we can define, through a classical measurable selection argument, the following function $u^\star$,
\begin{align}\label{eq:u*_CPT}
	u^\star (t,x,y,z,\gamma) \in \argmax_{u \in U} h_{\rm A} (t, x, y, z, \gamma, u), \; (t,x,y,z,\gamma) \in [0,T] \times \Omega \times \R \times \R^d \times \M^d,
\end{align}
so that the agent's best response to a contract $\xi = Y_T^{y_0,Z,\Gamma}$ can be written as $\nu^{\P^\star}_t := u^\star (t, X, Y_t^{y_0,Z,\Gamma}, Z_t, \Gamma_t)$, $t \in [0,T]$, $\P^\star$--a.s. One can also prove that the value of the agent's optimal control problem associated to such contract $\xi$ of the form \eqref{eq:contract_CPT} boils down to $y_0$, namely $V_{\rm A}(Y_T^{y_0,Z,\Gamma})=y_0$. Consequently, to ensure the agent's participation in the contractual agreement, it is clear that the principal must choose $y_0 \geq R_{\rm A}$.

\paragraph*{Principal's problem.} Under the \emph{restriction} to contracts of the form \eqref{eq:contract_CPT}, the principal's problem boils down to a more standard stochastic control problem of the following form:
\begin{align}\label{pb_principal_CPT_simple}
	\widetilde V_{\rm P} := \sup_{y_0 \geq R_{\rm A}} \underline V_{\rm P} (y_0), \; \text{ with } \;
	\underline V_{\rm P} (y_0) := \sup_{(Z,\Gamma) \in \Vc} \sup_{\P^\star \in \Pc^\star(Y_T^{y_0,Z,\Gamma})} J_{\rm P} \big(Y_T^{y_0,Z,\Gamma}, \P^\star \big).
\end{align}
In particular, for fixed $y_0 \geq R_{\rm A}$, the value $\underline V_{\rm P} (y_0)$ corresponds to the value of a stochastic control problem with two state variables, $X$ and $Y^{y_0,Z,\Gamma}$. Notice that, under $\P^\star \in \Pc^\star(Y_T^{y_0,Z,\Gamma})$, the process $X$ admits the following dynamics,
\begin{align}\label{eq:X-dynamics-CPT}
	\drm X_t = \sigma \big(t, X, \nu^{\P^\star}_t \big) \Big( \lambda \big(t, X, \nu_t^{\P^\star} \big) \mathrm{d}t + \mathrm{d}W_t \Big), \quad t\in[0,T],
\end{align}
corresponding to \cite[Equation (3.6)]{cvitanic2018dynamic}. Using this dynamics for $X$ under the optimal effort, one can also compute the dynamics satisfied by the process $Y^{y_0,Z,\Gamma}$ defined in \eqref{eq:contract_CPT} to obtain for all $t\in[0,T]$
\begin{align*}
	\drm Y_t^{y_0,Z,\Gamma} = &\ \bigg( Z_t \cdot [\sigma \lambda]\big(t, X, \nu_t^{\P^\star} \big) + \dfrac12 {\rm Tr} \Big[ \Gamma_t \big[ \sigma \sigma^\top \big] \big(t, X, \nu_t^{\P^\star} \big)  \Big] - \Hc_{\rm A} \big(t, X, Y^{y_0,Z,\Gamma}_t, Z_t, \Gamma_t \big) \bigg) \drm t
	+ Z_t \cdot \sigma \big(t, X, \nu_t^{\P^\star} \big) \mathrm{d}W_t,
\end{align*}
similarly to \cite[Equation (3.8)]{cvitanic2018dynamic}.
Using the definition of the Hamiltonian as well as the fact that the optimal effort is defined as one of its maximiser, 
we can actually further simplify the previous dynamics to obtain,
\begin{align}\label{eq:Y-dynamics-CPT}
	\drm Y_t^{y_0,Z,\Gamma} = &\ \Big( c_{\rm A} \big(t, X, \nu_t^{\P^\star} \big) + Y_t^{y_0,Z,\Gamma} k_{\rm A} \big(t, X, \nu_t^{\P^\star} \big)  \Big) \drm t
	+ Z_t \cdot \sigma \big(t, X, \nu_t^{\P^\star} \big) \mathrm{d}W_t, \quad t\in[0,T].
\end{align}
Again for future reference, it would be convenient to summarise this \emph{restricted} problem as follows.
\begin{problem}[Restriction to contracts of the form \eqref{eq:contract_CPT}]\label{pb:PA_solution_CPT}
Consider the standard stochastic control problem \eqref{pb_principal_CPT_simple},
\begin{align*}
	\widetilde V_{\rm P} := \sup_{y_0 \geq R_{\rm A}} \underline V_{\rm P} (y_0), \; \text{ with } \;
	\underline V_{\rm P} (y_0) := \sup_{(Z,\Gamma) \in \Vc} \sup_{\P^\star \in \Pc^\star(Y_T)} J_{\rm P} \big(Y_T,\P^\star \big).
\end{align*} 
where the state variable $(X,Y)$ is solution to the following system of {\rm SDE}s:
\begin{subequations}\label{eq:dynamics-CPT}
\begin{align}
	\drm X_t &= \sigma \big(t, X, \nu^{\P^\star}_t \big) \Big( \lambda \big(t, X, \nu^{\P^\star}_t \big) \mathrm{d}t + \mathrm{d}W_t \Big), \qquad \qquad \qquad \qquad \quad \; \;  t\in[0,T], \\
    \drm Y_t &= \Big( c_{\rm A} \big(t, X, \nu^{\P^\star}_t \big) + Y_t \, k_{\rm A} \big(t, X, \nu^{\P^\star}_t \big) \Big)  \drm t
	+ Z_t \cdot \sigma \big(t, X, \nu^{\P^\star}_t \big) \mathrm{d}W_t, \quad t\in[0,T],
\end{align}
\end{subequations}
coupled through the agent's optimal effort $\nu^{\P^\star}_t := u^\star (t, X, Y_t, Z_t, \Gamma_t)$ under $\P^\star \in \Pc^\star(Y_T)$.
\end{problem}
In a Markovian setting, one can solve this stochastic control problem as usual, namely by writing the principal's Hamiltonian as well as the Hamilton-Jacobi-Bellman (HJB for short) equation satisfied by the value function $\underline V_{\rm P} (y_0)$ (see \cite[Theorem 3.9]{cvitanic2018dynamic}), and finally optimise it over $y_0 \geq R_{\rm A}$. More importantly for our purpose in this paper, one can already remark that the value of \Cref{pb:PA_solution_CPT} should be lower than the value in \Cref{pb:PA_original}, since considering contracts $\xi$ of the form \eqref{eq:contract_CPT} is a restriction of the original set $\Xi$ of admissible contracts. This result can be found in \cite[Proposition 3.4]{cvitanic2018dynamic}, and is a direct consequence of \cite[Proposition 3.3]{cvitanic2018dynamic} mentioned above. For further reference, this result is stated in the lemma below.
\begin{lemma}\label{lem:pb1>pb2}
    The original value defined by {\rm \Cref{pb:PA_original}} is greater than the value in {\rm \Cref{pb:PA_solution_CPT}}, namely $V_{\rm P} \geq \widetilde V_{\rm P}$.
\end{lemma}

\paragraph*{Optimality of the revealing contracts. } To summarise, when restricting the study to contracts $\xi \in \Xi$ of the form \eqref{eq:contract_CPT}, one can first easily determine the agent's optimal response to such contract, namely by computing the maximiser of his Hamiltonian. Then, still under this restriction, the principal's problem boils down to a more standard stochastic problem, namely \Cref{pb:PA_solution_CPT}.
These two steps are therefore quite straightforward using standard tools from stochastic control theory. However, the last and crucial step requires to prove that the restriction to revealing contracts is without loss of generality. More precisely, one need to argue that the value of the original principal's problem, namely $V_{\rm P}$ recalled in \Cref{pb:PA_original}, actually coincides with the value when restricting to revealing contracts, \textit{i.e.} $\widetilde V_{\rm P}$ in \Cref{pb:PA_solution_CPT}. While it is clear that $V_{\rm P} \geq \widetilde V_{\rm P}$, as mentioned in \Cref{lem:pb1>pb2} above, the main mathematical challenge is to show the converse inequality, namely $V_{\rm P} \leq \widetilde V_{\rm P}$. This result, stated in \cite[Theorem 3.6]{cvitanic2018dynamic}, relies on the sophisticated theory of second-order BSDEs. We deliberately omit to further develop this step here, as it is precisely the part that our new approach allows us to bypass. It nevertheless seems essential to stress that this is the most difficult stage, which probably prevents a wider application of principal--agent problems with volatility control, but also hampers natural extensions of these problems, as discussed in introduction. \modif{Surprisingly, our approach reveals that if \Cref{ass:duality} is not satisfied, the method in \cite{cvitanic2018dynamic} may fail to yield optimal contracts. We refer the reader to \Cref{rk:duality} for a more detailed discussion of this point, and to \Cref{ss:counter-example} for an illustrative example.}

\section{A new approach}\label{sec:new_approach}

In this section, we first start by formalising a `first-best' alternative problem, in which the principal, in addition to offering a contract, also controls the quadratic variation of the output process. We then provide in \Cref{ss:approach} a straightforward resolution approach for this alternative problem, relying on the theory of BSDEs only. 
We will finally prove in \Cref{ss:penalisation_contract} that this reformulation actually leads to the exact same solution compared to the original problem considered in \Cref{sec:model}. 

\subsection{The `first-best' formulation}\label{ss:fb_formulation}

As already mentioned, the approach developed by \citeayn{cvitanic2018dynamic} mostly relies on the observation that, if the principal has access to the trajectory of $X$, she can actually compute its quadratic variation $\langle X \rangle$ pathwise. In general, observing the quadratic variation would still prevent the principal from deducing the controls on the volatility, hence the relevance of considering contract of the form \eqref{eq:contract_CPT}, indexed on $X$ and $\langle X \rangle$ through sensitivity parameters $Z$ and $\Gamma$ to be optimised. The main idea behind our reformulation is to assume further that, since the principal can compute the quadratic variation $\langle X \rangle$, she could actually directly control it. In other words, the idea here is to study an alternative `first-best' formulation of the original principal--agent problem with volatility control, in which the principal enforces a contract $\xi$ as well as a quadratic variation characterised by its density process $\Sigma$. The agent should then optimise his effort in order to maximise his expected utility given such contract $\xi$ as before, but now the admissible set of efforts is restricted to those achieving the quadratic variation imposed by the principal. More precisely, we will consider in the reformulated problem that, given a contract $\xi$ and a quadratic variation density process $\Sigma$, the agent should optimally choose his effort $\nu$ on the drift and the volatility of the output process $X$ so that $\drm \langle X \rangle_t = \Sigma_t \drm t$ for all $t \in [0,T]$. Then, the principal's optimisation problem boils down to choosing the optimal contract $\xi$ and the optimal quadratic variation density process $\Sigma$.

\medskip

To consider the agent's problem in this `first-best' reformulation, we shall again fix a contract $\xi$, but also an appropriate process $\Sigma$, representing the density of the quadratic variation with respect to the Lebesgue measure. Both would be chosen optimally by the principal, anticipating the agent's best response. 
With this in mind, we define the admissible set of processes $\Sigma$ by
\begin{align}\label{eq:set_qvar}
	\Sc := \big\{\Sigma \; \F\text{--progressively measurable process, taking values in} \; \S^d_+ \},
\end{align}
recalling that $\S_+^{d}$ denotes the set of $d \times d$ positive semi-definite symmetric matrices with real entries. 
For any process $\Sigma \in \Sc$, we will denote by $\Sigma^{1/2}$ its $(d,d)$-dimensional symmetric square root, and by $\Sigma^{-1/2}$ its pseudo-inverse.

\medskip

For fixed $\Sigma \in \Sc$, the agent's admissible efforts are now limited to those that achieve the desired quadratic variation. 
Mathematically, the set of admissible controls should be modified as follows:
\begin{align}\label{def:admissible_ctrl_agent_FB}
	\Uc^\circ(\Sigma) &:= \big\{\nu \in \Uc \; : \; \nu_t \in U_t^\circ(X,\Sigma_t), \; t \geq 0\}, \nonumber \\ \text{ with } \; 	U_t^\circ(x,S) &:= \big\{ u \in U \; : \; \big[\sigma \sigma^\top \big] (t,x,u) = S \big\}, \quad (t,x,S) \in [0,T] \times \Omega \times \S^d_+.
\end{align}
Accordingly, the set $\Pc^\circ(\Sigma)$ of admissible probability measures is defined by:
\begin{align*}
	\Pc^\circ(\Sigma) := \big\{ \P \in \Pc \; : \; \big[\sigma \sigma^\top \big] (t,X,\nu_t^\P) = \Sigma_t, \; \P\text{--a.s.}\}.
\end{align*}

\begin{remark}\label{rk:weakformulation_FB}
	In this framework, as the density of the quadratic variation is fixed to some process $\Sigma \in \Sc$, one can actually write the weak formulation for the agent's problem by means of Girsanov's theorem. More precisely, one can originally define the canonical process $X$ through the following {\rm SDE},
	\begin{align}\label{eq:SDE_qvar}
		X_t = x_0 + \int_0^t \Sigma_s^{1/2} \mathrm{d}W^\circ_s, \quad t \in [0,T],
	\end{align}
	where $W^\circ$ is a $d$-dimensional Brownian Motion. We can show that starting from a weak solution $\widebar \P \in \Pc$ to \eqref{eq:SDE_qvar}, one can construct, for any $\nu \in \Uc^\circ(\Sigma)$, a weak solution $\P^\nu \in \Pc^\circ(\Sigma)$ to the original {\rm SDE} \eqref{eq:SDE_drift} through absolute continuous change of measure. 

    \medskip

    We proceed in two steps. We first relate the previous {\rm SDE} to the following one:
    \begin{align}\label{eq:SDE_driftless}
		X_t = x_0 + \int_0^t \sigma (s, X,\nu_s) \mathrm{d}W_s, \quad t \in [0,T],
	\end{align}
    where $W$ is a $n$-dimensional Brownian Motion. For this, we consider an additional $(n-d)$-dimensional Brownian Motion $W^1$, potentially defined on an extended probability space, such that $\widebar W := (W^\circ, W^1)$ is an $n$-dimensional Brownian Motion. Moreover, denote by $\Rc_n$ the set of all matrices of rotations of $\M^n$, \textit{i.e.} $R \in \M^n$ such that $RR^\top = I_n$. For any $R \in \Rc_n$, we will further denote by $R^\circ$ the $(d,n)$-dimensional matrix consisting of the $d$ first rows of $R$, and $R^1$ the $(n-d,n)$-dimensional matrix consisting of the $n-d$ remaining rows of $R$. 
    For a given control process $\nu \in \Uc^\circ(\Sigma)$, let $(R_t)_{t \in [0,T]}$ be an $\Rc_n$-valued optional process with
	\begin{align*}
		R_t^\circ := \Sigma_t^{-1/2} \sigma(t,X,\nu_t), \quad t \in [0,T],
	\end{align*}
    where $\Sigma_t^{-1/2}$ denotes the pseudo-inverse of $\Sigma_t^{1/2}$ for all $t \in [0,T]$, and $(R^1_t)_{t \in [0,T]}$ arbitrary. By Levy's characterisation of Brownian Motion, the process $W$ defined by $\drm W_t = R_t^{\top} \drm \widebar W_t$ for all $t \in [0,T]$ is also a $\widebar \P$--Brownian Motion, and {\rm SDE} \eqref{eq:SDE_qvar} can be rewritten as
	\begin{align*}
		X_t = x_0 + \int_0^t \Sigma_s^{1/2} \mathrm{d}W^\circ_s
		= x_0 + \int_0^t \Sigma_s^{1/2} R_s^\circ \mathrm{d}W_s
		= x_0 + \int_0^t \sigma (s, X,\nu_s) \mathrm{d}W_s, \quad t \in [0,T].
	\end{align*}
    In other words, any weak solution $\widebar \P$ to {\rm SDE} \eqref{eq:SDE_qvar} is also a weak solution to {\rm SDE} \eqref{eq:SDE_driftless}. Note that the converse also holds, simply by Levy's characterisation of Brownian Motion. 

    \medskip
 
    Consider then the following exponential local martingale
    \begin{align}\label{eq:girsanov}
		M^\nu_t := \Ec_t \bigg( \int_0^\cdot \lambda(s,X,\nu_s) \drm W_s \bigg), \quad t \in [0,T],
	\end{align}
	where $\Ec$ denotes the Doleans--Dade exponential, namely
	\begin{align*}
		\Ec_t(M) := \exp \bigg( M_t - \dfrac12 \langle M \rangle_t \bigg), \quad t \in [0,T],
	\end{align*}
	for any local martingale $M$. 
    As $\lambda$ is assumed to be bounded, the process $M^\nu$ defined above is actually a martingale, and one can thus apply Girsanov's theorem. In particular, by defining the following change of probability,
	\begin{align*}
		\dfrac{\drm \P^\nu}{\drm \widebar \P} \bigg|_{\Fc_T} 
		= M^\nu_T,
	\end{align*}
    the process $W^\nu$ defined by
	\begin{align*}
		W^{\nu}_\cdot := W_\cdot - \int_0^\cdot \lambda (s,X,\nu_s) \drm s
    \end{align*}
	is a $n$-dimensional $\P^\nu$--Brownian Motion, and $X$ admits the following representation:
	\begin{align*}
		X_t = x_0 + \int_0^t [\sigma \lambda] (s,X,\nu_s) \drm s
		+ \int_0^t \sigma (s,X,\nu_s) \mathrm{d}W^{\nu}_s, \quad t \in [0,T].
	\end{align*}
	Therefore, starting from a weak solution $\widebar \P$ to \eqref{eq:SDE_driftless}, any admissible control process $\nu \in \Uc^\circ(\Sigma)$ induces a weak solution $\P^\nu \in \Pc^\circ(\Sigma)$ to {\rm SDE} \eqref{eq:SDE_drift}. The converse is also true, simply by inverting the above construction. 
\end{remark}
Informally, the previous remark highlights that, when the quadratic variation is fixed, then the agent's impact on the output process $X$ can be defined as equivalent change of measures through Girsanov's theorem. In other words, for fixed $\Sigma \in \Sc$, the set $\Pc^\circ(\Sigma)$ is dominated by $\widebar \P$, weak solution to \eqref{eq:SDE_driftless}. Note that $\widebar \P$ obviously depends on the original choice of $\Sigma \in \Sc$, but for simplicity we decide to omit this dependency in the notations, especially since $\Sigma \in \Sc$ will be fixed to solve the agent problem. 

\begin{remark}
The previous remark is also mentioned in {\rm \cite[Section 4.4]{cvitanic2018dynamic}}. Indeed, in order to deal with volatility control in the original stochastic control problem for the agent, the idea is to first isolate the control of the quadratic variation, thus considering the following reformulation of the agent's problem:
    \begin{align*}
	   V_{\rm A} (\xi) = \sup_{\Sigma \in \Sc} \sup_{\P \in \Pc^\circ(\Sigma)} J_{\rm A} (\xi,\P).
    \end{align*}
For fixed $\Sigma \in \Sc$, all probability measures $\P \in \Pc^\circ(\Sigma)$ are equivalent through Girsanov's theorem, as mentioned in the previous remark. One can therefore represent the value function of the agent for fixed quadratic variation as a backward stochastic differential equation. Then, since the agent can optimise over admissible quadratic variations, his problem becomes a supremum of {\rm BSDE}s, \textit{i.e.} a second-order {\rm BSDE}. While we will follow the exact same first step, namely to represent the agent's continuation utility as a {\rm BSDE}, we will not have to deal with the second step as in our `first-best' problem, the quadratic variation is chosen by the principal, not delegated to the agent.
\end{remark}

We can finally formalise the principal--agent problem in this `first-best' scenario. More precisely, given a contract $\xi$ and a quadratic variation density process $\Sigma \in \Sc$, the optimisation problem faced by the agent is defined by
\begin{align}\label{eq:pb_agent_FB}
	V^\circ_{\rm A} (\xi,\Sigma) := \sup_{\P \in \Pc^\circ(\Sigma)} J_{\rm A} (\xi,\P).
\end{align}
In parallel to \eqref{def:contract_set} and \eqref{def:optimal_response}, for fixed $\Sigma \in \Sc$, we define the set $\Xi^\circ(\Sigma)$ of admissible contracts as
\begin{align}\label{def:contract_set_FB}
	\Xi^\circ(\Sigma) := \big\{\Fc_T\text{-measurable random variable $\xi$, satisfying \eqref{eq:integrability_contract_agent} and} \; V^\circ_{\rm A} (\xi,\Sigma) \geq R_{\rm A} \big\}.
\end{align}
and denote by $\Pc^{\circ,\star}(\xi,\Sigma)$ the set of agent's optimal response.

\medskip

Then, the goal of the principal is to maximise her objective function by choosing the optimal quadratic variation process $\Sigma \in \Sc$ as well as the optimal contract $\xi \in \Xi^\circ(\Sigma)$, anticipating the agent's optimal response:
\begin{align}\label{eq:pb_principal_FB}
	V_{\rm P}^\circ := \sup_{\Sigma \in \Sc} \sup_{\xi \in \Xi^\circ(\Sigma)} \sup_{\P^\star \in \Pc^{\circ,\star}(\xi,\Sigma)} J_{\rm P} (\xi,\P^\star).
\end{align}
Recall that by convention, the supremum over an empty set is equal to $- \infty$. In particular, if the principal suggests a quadratic variation which is not implementable by the agent, then the set $\Pc^{\circ,\star}(\xi,\Sigma)$ would be empty, leading to a utility of $- \infty$ for the principal, which cannot be optimal. Therefore, we can assume without loss of generality that 
for any potentially optimal $\Sigma \in \Sc$, there exists an admissible effort $\nu \in \Uc^\circ(\Sigma)$, and thus a corresponding probability measure $\P^\nu \in \Pc^{\circ}(\Sigma)$, such that for all $t \in [0,T]$, $\Sigma_t = [\sigma \sigma^\top] (t,X,\nu_t), \; \P^\nu$--a.s. This implies in particular that for all $t \in [0,T]$, $\Sigma_t$ should take values in $\S_t(X)$, defined for all $x \in \Omega$ by
\begin{align}\label{eq:set_Sigma}
	\S_t(x) := \big\{S \in \S_+^{d} \; : \; S = \big[\sigma \sigma^\top \big] (t,x,u) \; \text{ for some } u \in U \big\}.
\end{align}
Furthermore, recall that the original set $\Uc$ of admissible controls is defined in \Cref{ass:weak_uniqueness} such that, for any control $\nu \in \Uc$, there exists a unique weak solution to the controlled SDE \eqref{eq:SDE_drift}. In particular, the set $\Uc^\circ(\Sigma)$ of admissible (constrained) controls also satisfies this property. Therefore, under the previous restriction that $\Sigma \in \Sc$ should be such that $\Uc^\circ(\Sigma) \neq \varnothing$, it is clear that weak uniqueness also holds for SDE~\eqref{eq:SDE_qvar}. In other words, for fixed $\Sigma \in \Sc$, the weak solution $\widebar \P$ to SDE~\eqref{eq:SDE_qvar} defined in \Cref{rk:weakformulation_FB} is actually unique, and therefore admits the predictable martingale representation property (see for example \citeayn[Theorem III.4.29]{jacod2003limit}). 

\begin{remark}\label{rk:MRP2}
    This assumption that weak uniqueness holds for {\rm SDE}~\eqref{eq:SDE_drift}, and thus for {\rm SDE}~\eqref{eq:SDE_qvar}, which is not enforced in {\rm \cite{cvitanic2018dynamic}}, is considered here to avoid technical considerations regarding the martingale representation property. Indeed, without weak uniqueness, the solution $\widebar \P$ to {\rm SDE}~\eqref{eq:SDE_qvar} may not admit the predictable martingale representation property. More precisely, the decomposition of any $(\F,\widebar \P)$-local martingale with respect to $X$ would in general involve an extra martingale term, orthogonal to $X$. As this martingale representation property is used to derive the optimal form of contracts, such an additional martingale term would complicate the study. In {\rm \cite{cvitanic2018dynamic}}, this additional martingale term can be hidden in the component `$K$' of the solution $(Y,Z,K)$ to the appropriate second-order {\rm BSDE}, as the set of probability measures considered is naturally saturated {\rm(}see {\rm \citeayn[Definitions 5.1 and 5.2]{possamai2018stochastic})}. We refer back to {\rm \Cref{rk:MRP1}} for a discussion on the reasonableness of this assumption.
\end{remark}


Again for future reference, we can summarise this \emph{alternative} `first-best' problem as follows.
\begin{problem}\label{pb:PA_FB_reformulation}
The `first-best' reformulation of the original principal--agent problem is defined by
\begin{align}
    V_{\rm P}^\circ := \sup_{\Sigma \in \Sc} \sup_{\xi \in \Xi^\circ(\Sigma)} \sup_{\P^\star \in \Pc^{\circ,\star}(\xi,\Sigma)} J_{\rm P} (\xi,\P^\star),
\end{align}
where $\Pc^{\circ,\star}(\xi,\Sigma) := \{ \P^\star \in \Pc^\circ(\Sigma) \; : \; V^\circ_{\rm A} (\xi,\Sigma) = J_{\rm A} (\xi, \P^\star) \} \neq \varnothing$ for $\Sigma \in \Sc$ and $\xi \in \Xi^\circ(\Sigma)$.
\end{problem}
We first provide in \Cref{ss:approach} a very standard approach, relying on BSDEs theory, to characterise the optimal form of contracts for this `first-best' alternative problem. The result is given in \Cref{thm:solution_FB}. We then prove the main result of this paper, namely \Cref{thm:main}, which states the equivalence between \Cref{pb:PA_original} and \Cref{pb:PA_FB_reformulation}, \textit{i.e.} $V_{\rm P} = V_{\rm P}^\circ$. 
Before proceeding, one can already notice that the principal's value $V_{\rm P}^\circ$ in the alternative `first-best' problem summarised above, namely \Cref{pb:PA_FB_reformulation}, is greater than the value $V_{\rm P}$ in the original problem (\Cref{pb:PA_original}). This result is stated in the following \Cref{lem:pb3>pb1}, whose proof mostly relies on the observation that, in \Cref{pb:PA_FB_reformulation}, the principal has an additional control compared to the original problem.

\begin{lemma}\label{lem:pb3>pb1}
    The value of {\rm \Cref{pb:PA_FB_reformulation}} is greater than the value of {\rm \Cref{pb:PA_original}}, \textit{i.e.} 
    $V_{\rm P}^\circ \geq V_{\rm P}$.
\end{lemma}

\begin{proof}[\Cref{lem:pb3>pb1}]
Recall that we can assume, without loss of generality, that any potentially optimal contract $\xi \in \Xi$ for \Cref{pb:PA_original} should be such that $\Pc^\star(\xi) \neq \varnothing$, otherwise $V_{\rm P} = - \infty$. One can thus consider $\P^\star \in \Pc^\star(\xi)$, and its corresponding optimal effort process $\nu^{\P^\star} \in \Uc$, leading to the following representation for the output process $X$:
\begin{align*}
	X_t = x_0 + \int_0^t \sigma \big(s, X,\nu^{\P^\star}_s \big) \big( \lambda \big(s, X,\nu^{\P^\star}_s \big) \mathrm{d}s +  \mathrm{d}W_s \big), \; t\in[0,T],\; \P^\star \textnormal{--a.s.}
\end{align*}
This naturally implies that the quadratic variation of $X$, namely $\langle X \rangle$, observable by the principal, satisfies 
\begin{align*}
    \langle X \rangle_t = \int_0^t \Sigma^{\P^\star}_s \drm s, \quad \text{where} \quad
    \Sigma^{\P^\star}_t := \big[ \sigma \sigma^\top \big] \big(t, X,\nu^{\P^\star}_t \big), \; t \in [0,T], \; \P^\star \textnormal{--a.s.}.
\end{align*}
Note that the process $\Sigma^{\P^\star}$ defined above belongs to the set $\Sc$ by definition and we have $\nu^{\P^\star} \in \Uc^\circ(\Sigma^{\modifreview{\P^\star}})$, or equivalently, $\P^\star \in \Pc^\circ(\Sigma^{\P^\star})$. Additionally, by definition of $\P^\star \in \Pc^\star(\xi)$ in \eqref{def:optimal_response}, we have $V_{\rm A} (\xi) = J_{\rm A} (\xi,\P^\star) \geq R_{\rm A}$, and therefore 
\begin{align*}
    R_{\rm A} \leq V_{\rm A} (\xi) = J_{\rm A} (\xi,\P^\star) \leq \sup_{\P \in \Pc^\circ(\Sigma^{\P^\star})} J_{\rm A} (\xi,\P) =: V_{\rm A}^\circ \big(\xi, \Sigma^{\P^\star} \big),
\end{align*}
implying in particular that $\xi \in \Xi^\circ(\Sigma^{\P^\star})$. On the other hand, since we naturally have $V_{\rm A} (\xi) \geq V_{\rm A}^\circ (\xi, \Sigma)$ for all $\Sigma \in \Sc$ and $\xi \in \Xi$, we can also deduce from the previous inequality that $\P^\star \in \Pc^{\circ,\star}(\xi,\Sigma^{\P^\star})$. To summarise, for any $\xi \in \Xi$ and $\P^\star \in \Pc^\star(\xi)$, there exists $\Sigma^{\P^\star} \in \Sc$ such that $\xi \in \Xi^\circ(\Sigma^{\P^\star})$ and $\P^\star \in \Pc^{\circ,\star}(\xi,\Sigma^{\P^\star})$, which implies
\begin{align*}
    \sup_{\xi \in \Xi} \sup_{\P^\star \in \Pc^{\star}(\xi)} J_{\rm P} (\xi,\P^\star) \leq 
    \sup_{\Sigma \in \Sc} \sup_{\xi \in \Xi^\circ(\Sigma)} \sup_{\P^\star \in \Pc^{\circ,\star}(\xi,\Sigma)} J_{\rm P} (\xi,\P^\star)
\end{align*}
and therefore concludes the proof.
\end{proof}

\subsection{Solving the `first-best' problem}\label{ss:approach}

While \Cref{pb:PA_original} is proved to be equivalent to \Cref{pb:PA_solution_CPT} in \cite{cvitanic2018dynamic} using the theory of 2BSDEs, we demonstrate in this section that the alternative `first-best' problem, namely \Cref{pb:PA_FB_reformulation} introduced above, can also be reformulated into a more standard stochastic control problem, defined as \Cref{pb:first-best_solution} below, but relying on BSDEs only. 

\medskip

To be consistent with the approach developed in \cite{cvitanic2018dynamic}, we first introduce the appropriate Hamiltonian for the agent when the quadratic variation is fixed, as well as the relevant form of contracts for the reformulated problem \eqref{eq:pb_principal_FB}. With this in mind, let $(t,x,y,z,S) \in [0,T] \times \Omega \times \R \times \R^d \times \S^d_+$, and consider the following Hamiltonian:
\begin{align}\label{def_hamiltonian_FB}
	\Hc^\circ_{\rm A} (t, x, y, z, S) &:= \sup_{u \in U_t^\circ(x,S)} \; h^\circ_{\rm A} (t, x, y, z, u), \\ \text{ with } \; 
	h^\circ_{\rm A} (t, x, y, z, u) &:= [\sigma \lambda ] (t, x, u) \cdot z - c_{\rm A}(t, x, u) - k_{\rm A}(t,x,u) y, \quad u \in U_t^\circ(x,S), \nonumber
\end{align} 
recalling that the set $U_t^\circ(x,S)$ is defined by \eqref{def:admissible_ctrl_agent_FB}.

\begin{remark}\label{rk:hamiltonian}
	The Hamiltonian defined above in \eqref{def_hamiltonian_FB} is a simplified version of the original Hamiltonian in \eqref{def_hamiltonian_CPT}. Indeed, we consider here that the quadratic variation is fixed to some level through its density process $\Sigma$, and in particular, cannot be optimised by the agent. In other words, at each time $t \in [0,T]$, the agent has to choose his effort $\nu_t$ taking values in $U$ and satisfying $[\sigma \sigma^\top] (t,X,\nu_t) = \Sigma_t$. Under this constraint, the `$\gamma$ part' of $h_{\rm A}$ defined in \eqref{def_hamiltonian_CPT} becomes a constant, uncontrolled by the agent. Mathematically speaking, we have:
	\begin{align*}
		\sup_{u \in U_t^\circ(x,S)} h_{\rm A} (t, x, y, z, \gamma, u) 
		&= \sup_{u \in U_t^\circ(x,S)} h^\circ_{\rm A} (t, x, y, z, u) + \dfrac12 {\rm Tr} \big[ \gamma S \big]. 
	\end{align*}
	The previous equality also highlights the following relationship between the two Hamiltonians, 
	\begin{align}\label{eq:link_hamiltonian}
		\Hc_{\rm A} (t, x, y, z, \gamma) &= \sup_{S \in \S_t(x) } \bigg\{ \Hc^\circ_{\rm A} (t, x, y, z, S) + \dfrac12 {\rm Tr} \big[ \gamma S \big] \bigg\}.
	\end{align}
	In other words, $2 \Hc_{\rm A} (\cdot, \gamma)$ is the convex conjugate, or Legendre--Fenchel transformation, of $-2 \Hc^\circ_{\rm A} (\cdot, S)$.
\end{remark}
	
In our reformulated problem, because the quadratic variation is originally fixed by the principal, the relevant form of contract will actually be similar to the one considered in principal--agent problems with drift control only. More precisely, we are led to consider here $\xi = Y_T^{y_0,Z}$, where the process $Y^{y_0,Z}$ is defined as the solution to the following \modifreview{SDE},
\begin{align}\label{eq:contract_FB}
	Y_t^{y_0,Z} = y_0 - \int_0^t \Hc_{\rm A}^\circ \big(s, X, Y^{y_0,Z}_s, Z_s, \Sigma_s \big) \drm s + \int_0^t Z_s \cdot \drm X_s, \quad \P\textnormal{--a.s.}, \; \text{for all } \P \in \Pc^\circ (\Sigma),
\end{align}
for a constant $y_0 \in \R$ and a process $Z$ satisfying appropriate integrability conditions. In parallel with \Cref{def:contrat_vol}, we define the set $\Vc^\circ(\Sigma)$ of appropriate processes $Z$ as follows. 
\begin{definition}\label{def:contract_FB}
	Let $y_0 \in \R$ and $\Sigma \in \Sc$. For any $\F$--predictable processes $Z$ taking values in $\R^d$, define the process $Y^{y_0,Z}$ using \eqref{eq:contract_FB}. We will denote $Z \in \Vc^\circ(\Sigma)$ if moreover
	\begin{enumerate}[label=$(\roman*)$]
		\item $\| Z \|^p_{\H^p} < \infty$ and $\| Y^{y_0,Z} \|^p_{\D^p} < \infty$ for some $p > 1;$
		\item there exists $\P \in \Pc^\circ(\Sigma)$ such that
		\begin{align}\label{eq:optimal_effort_FB}
			\Hc^\circ_{\rm A} \big(t, X, Y_t^{y_0,Z}, Z_t, \Sigma_t \big) = h^\circ_{\rm A} \big(t, X, Y_t^{y_0,Z}, Z_t, \nu_t^{\P} \big), \quad \drm t \otimes \P\text{--a.e. on } [0,T] \times \Omega.
		\end{align}
	\end{enumerate}
\end{definition}
For fixed $\Sigma \in \Sc$, any contract defined by $\xi = Y_T^{y_0,Z}$ through \eqref{eq:contract_FB} for some $y_0 \in \R$ and $Z \in \Vc^\circ(\Sigma)$ is an $\Fc_T$-measurable random variable satisfying the integrability condition \eqref{eq:integrability_contract_agent}. If this contract satisfies in addition the agent's participation constraint, we conclude that $\xi \in \Xi^\circ(\Sigma)$. Conversely, we will prove below in \Cref{thm:solution_FB} that the restriction to contracts of the form $\xi = Y_T^{y_0,Z}$, for $(y_0,Z) \in [R_{\rm A},\infty) \times \Vc^\circ(\Sigma)$, is without loss of generality.

\begin{remark}
	Intuitively, when the density $\Sigma \in \Sc$ of the quadratic variation is fixed, the new form of contracts \eqref{eq:contract_FB} can be derived from the general form of contracts \eqref{eq:contract_CPT}. Indeed, starting from \eqref{eq:contract_CPT} under $\P \in \Pc^\circ(\Sigma)$, and using in particular {\rm \Cref{rk:hamiltonian}} on the Hamiltonian, we derive the following representation,
	\begin{align*}
		Y_t^{y_0,Z,\Gamma} 
		&= y_0 - \int_0^t \bigg( \Hc^\circ_{\rm A} \big(s, X, Y^{y_0,Z,\Gamma}_s, Z_s, \Sigma_s \big) + \dfrac12 {\rm Tr} \big[ \Gamma_s \Sigma_s \big] \bigg) \drm s + \int_0^t Z_s \cdot \drm X_s + \dfrac12 \int_0^t {\rm Tr} \big[ \Gamma_s  \Sigma_s \big] \drm s, \; t \in [0,T],
	\end{align*}
	thus naturally leading to \eqref{eq:contract_FB}, since the part of the contract indexed by $\Gamma$ simplifies.
\end{remark}

Before stating the main result of this section, we first solve the agent's problem \eqref{eq:pb_agent_FB}, in order to confirm the intuition on the relevant contract form \eqref{eq:contract_FB} highlighted by the previous remark. The proof of this result only relies on the theory of BSDEs, as opposed to 2BSDEs, needed to characterise the agent's problem originally defined in \eqref{eq:pb_agent} (see \cite[Proposition 4.6]{cvitanic2018dynamic}).

\begin{proposition}\label{prop:solve_agent_pb}
Let $\Sigma \in \Sc$ and $\xi \in \Xi^\circ(\Sigma)$. There exist a unique $\Fc_0^{\widebar \P+}$-measurable random variable $Y_0$, satisfying $\E^{\widebar \P}\left[Y_0\right] \geq R_{\rm A}$, and a unique $Z \in \H^p$ such that $\xi = Y_T$, for $Y \in \D^p$ defined as the solution to the following \modifreview{\rm SDE:}
\begin{align}\label{eq:ODE_FB}
    Y_t = Y_0 - \int_0^t \Hc_{\rm A}^\circ \big(s, X, Y_s, Z_s, \Sigma_s \big) \drm s + \int_0^t Z_s \cdot \drm X_s, \; t \in [0,T], \; \widebar \P\textnormal{--a.s.}
\end{align}
Moreover, 
\begin{enumerate}[label=$(\roman*)$]
    \item $V_{\rm A}^\circ (\xi,\Sigma) = \E^{\widebar \P}\left[Y_0\right] \geq R_{\rm A};$
    \item $\P^\circ \in \Pc^{\circ,\star}(\xi,\Sigma)$ if and only if $\Hc^\circ_{\rm A} (t, X, Y_t, Z_t, \Sigma_t) = h^\circ_{\rm A} (t, X, Y_t, Z_t, \nu_t^{\P^\circ})$, $\drm t \otimes \P^\circ\text{--a.e. on } [0,T] \times \Omega$.
\end{enumerate}
\end{proposition}

\begin{proof}[\Cref{prop:solve_agent_pb}]
\emph{Step 1.} 
We fix $\Sigma \in \Sc$ and $\xi \in \Xi^\circ(\Sigma)$ to consider the following BSDE
\begin{align}\label{eq:BSDE_proof}
    Y_t = \xi + \int_t^T \Hc_{\rm A}^\circ \big(s, X, Y_s, Z_s, \Sigma_s \big) \drm s - \int_t^T Z_s \cdot \drm X_s, \quad t \in [0,T] \quad \widebar \P\textnormal{--a.s.},
\end{align}
and argue that there exists a unique solution $(Y,Z) \in \D^p \times \H^p$ for some $p > 1$, to deduce \eqref{eq:ODE_FB}.
For this, we let $p_1 >1$ and $p_2 >1$ such that \eqref{eq:integrability_cost} and \Cref{eq:integrability_contract_agent} hold, respectively. For all $p^{\prime} \in (1,p_1)$, 
using Jensen's inequality and changing the probability from $\widebar \P$ to $\P^\nu$ using $M^\nu$ defined in \eqref{eq:girsanov}, we have
\begin{align*}
    \E^{\widebar \P} \Bigg[ \bigg(\int_0^T \big| \Hc^\circ_{\rm A} (t, X, 0,0,\Sigma_t) \big| \drm t\bigg)^{p^{\prime}} \Bigg]
    &\leq \E^{\widebar \P} \bigg[ \int_0^T \big| \Hc^\circ_{\rm A} (t, X, 0,0,\Sigma_t) \big|^{p^{\prime}} \drm t \bigg]
    = \E^{\widebar \P} \bigg[ \int_0^T \inf_{u \in U^\circ(\Sigma)} \big| c_{\rm A} (t, X, u) \big|^{p^{\prime}} \drm t \bigg]\\
    &= \E^{\P^{\nu}} \bigg[ M_T^\nu \int_0^T \inf_{u \in U^\circ(\Sigma)} \big| c_{\rm A} (t, X, u) \big|^{p^{\prime}} \drm t \bigg] \\
    &\leq \E^{\P^{\nu}} \Bigg[ \bigg( \int_0^T \inf_{u \in U^\circ(\Sigma)} \big| c_{\rm A} (t, X, u) \big|^{p^{\prime}} \drm t\bigg)^{q_1} \Bigg]^{\frac{1}{q_1}} \E^{\P^{\nu}} \Big[ \big| M_T^\nu\big |^{q_2} \Big]^{\frac{1}{q_2}},
\end{align*} 
where the last inequality comes from Holder's inequality with $q_1 := p_1/p^\prime > 1$ and $q_2 > 1$ its Holder's conjugate, \textit{i.e.} such that $\frac{1}{q_1}+\frac{1}{q_2}=1$. Since $\lambda$ is bounded, there exists a constant $C > 0$ such that
\begin{align*}
    \E^{\widebar \P} \Bigg[ \bigg(\int_0^T \big| \Hc^\circ_{\rm A} (t, X, 0,0,\Sigma_t) \big| \drm t\bigg)^{p^{\prime}} \Bigg]
    &\leq C \; \E^{\P^{\nu}} \bigg[ \int_0^T \inf_{u \in U^\circ(\Sigma)} \big| c_{\rm A} (t, X, u) \big|^{p_1} \drm t \bigg]^{\frac{1}{q_1}}
    \leq C \sup_{\P \in \Pc} \E^{\P} \bigg[ \int_0^T \big| c_{\rm A} (t, X, \nu_t) \big|^{p_1} \drm t \bigg]^{\frac{1}{q_1}},
\end{align*} 
which is finite since \eqref{eq:integrability_cost} holds for $p_1 > 1$. By similar arguments, we deduce from the fact that \Cref{eq:integrability_contract_agent} is satisfied for $p_2 > 1$ that $\E^{\widebar \P}[ |\xi |^{p^{\prime \prime}}] < + \infty$, for $p^{\prime \prime} \in (1, p_2)$. We thus obtain the following integrability condition,
\begin{align*}
    \E^{\widebar \P} \bigg[ |\xi |^{p} + \bigg(\int_0^T \big| \Hc^\circ_{\rm A} (t, X, 0,0,\Sigma_t) \big| \drm t\bigg)^{p} \bigg] < \infty, \quad p \in (1, p_1 \wedge p_2).
\end{align*}
Finally, by the boundedness of $\sigma, \lambda$ and $k_{\rm A}$, the Hamiltonian $\Hc_{\rm A}^\circ$ is uniformly Lipschitz-continuous in $(y,z)$. Since the assumptions of \citeayn[Theorem 4.2]{briand2003solutions} are satisfied (see also \citeayn[Theorem 4.1]{bouchard2018unified} for a more recent result), and $\widebar \P$ satisfies the predictable martingale representation property, we deduce that there exists a unique solution $(Y,Z) \in \D^p \times \H^p$ to BSDE \eqref{eq:BSDE_proof}, for some $p > 1$. In particular, such solution satisfies \Cref{def:contract_FB} $(i)$ and $Y_0$ is a $\Fc_0^{\widebar \P+}$-measurable random variable.

\medskip

\emph{Step 2.} We now prove the equality $(i)$, namely $V^\circ_{\rm A}(\xi,\Sigma) = \E^{\widebar \P}[Y_0]$. As a direct consequence, we will obtain that $Z \in \H^p$, introduced in Step 1, satisfies \eqref{eq:optimal_effort_FB} in \Cref{def:contract_FB}, implying that $(ii)$ is verified. By Step 1, any contract $\xi \in \Xi^\circ(\Sigma)$ can be written as
\begin{align*}
    \xi = Y_T := Y_0 - \int_0^T \Hc_{\rm A}^\circ \big(s, X, Y_s, Z_s, \Sigma_s \big) \drm s + \int_0^T Z_s \cdot \drm X_s, \quad \widebar \P\textnormal{--a.s.}
\end{align*}
Writing the dynamics under some arbitrary admissible effort $\nu \in \Uc^\circ(\Sigma)$, \textit{i.e.} under $\P^\nu \in \Pc^\circ(\Sigma)$, we have:
\begin{align*}
    \drm Y_t = \big( - \Hc_{\rm A}^\circ \big(t, X, Y_t, Z_t, \Sigma_t \big) + Z_t \cdot [\sigma \lambda] (t,X,\nu_t) \big) \drm t + Z_t \cdot \sigma  (t,X,\nu_t) \drm W_t, \quad \P^\nu\textnormal{--a.s.}
\end{align*}
Applying Ito's formula and using the definition of $h^\circ_{\rm A}$, one can compute, still under $\P^\nu \in \Pc^\circ(\Sigma)$,
\begin{align*}
    \Kc^{\P^\nu}_{\rm A}(T) Y_T - \int_0^T \Kc^{\P^\nu}_{\rm A}(t) c_{\rm A}(t,X,\nu_t) \drm t = Y_0 &- \int_0^T \Kc^{\P^\nu}_{\rm A}(t) \big( \Hc^\circ_{\rm A}(t,X,Y_t,Z_t,\Sigma_t) - h^\circ_{\rm A} (t, X, Y_t, Z_t, \nu_t) \big) \drm t\\
&+ \int_0^T \Kc^{\P^\nu}_{\rm A} (t)  Z_t \cdot \sigma (t,X,\nu_t) \drm W_t.
\end{align*}
Using this representation in the agent's objective function, and noticing that, since $Z \in \H^p$ and $k_{\rm A}, \sigma$ are bounded, the expectation of the stochastic integral term vanishes, we obtain:
\begin{align*}
    J_{\rm A} \big(\xi,\P^{\nu} \big) 
    &=\E^{\P^{\nu}}\big[Y_0\big] - \E^{\P^{\nu}} \bigg[\int_0^T \Kc^{\P^\nu}_{\rm A}(t) \big( \Hc^\circ_{\rm A} (t,X,Y_t,Z_t,\Sigma_t) - h^\circ_{\rm A} (t, X, Y_t, Z_t, \nu_t) \big) \drm t\bigg].
\end{align*}
Using the Girsanov's change of probability highlighted in \Cref{rk:weakformulation_FB}, we observe that
\begin{align*}
    \E^{\P^{\nu}}\big[ Y_0 \big] = \E^{\widebar \P}\big[M_T^\nu Y_0\big] = \E^{\widebar \P}\big[M_0^\nu Y_0\big] = \E^{\widebar \P}\big[Y_0\big],
\end{align*}
since $M^\nu$ is both an $(\F^{\widebar \P},\widebar \P)$- and an $(\F^{\widebar \P +},\widebar \P)$-martingale (see \cite[Proposition 2.2]{neufeld2014measurability}), and thus
\begin{align}\label{eq:reward}
    J_{\rm A} \big(\xi,\P^{\nu} \big) 
    = \E^{\widebar \P}\big[Y_0\big] - \E^{\P^{\nu}} \bigg[\int_0^T \Kc^{\P^\nu}_{\rm A}(t) \big( \Hc^\circ_{\rm A} (t,X,Y_t,Z_t,\Sigma_t) - h^\circ_{\rm A} (t, X, Y_t, Z_t, \nu_t) \big) \drm t\bigg].
\end{align}
On the one hand, it is clear by definition of $\Hc^\circ_{\rm A}$ and $h^\circ_{\rm A}$ that $J_{\rm A}(\xi,\P^{\nu}) \leq \E^{\widebar \P}[Y_0]$, from which follows the first inequality $V^\circ_{\rm A}(\xi,\Sigma) \leq \E^{\widebar \P}[Y_0]$. On the other hand, for all $\nu \in \Uc^\circ(\Sigma)$, we also have:
\begin{align*}
    V^\circ_{\rm A}(\xi,\Sigma) \geq J_{\rm A} \big(\xi,\P^{\nu} \big)
    = \E^{\widebar \P} [Y_0]-\E^{\P^{\nu}}\bigg[\int_0^T \Kc^{\P^\nu}_{\rm A}(t) \big( \Hc^\circ_{\rm A}(t,X,Y_t,Z_t,\Sigma_t)-h^\circ_{\rm A} (t, X, Y_t, Z_t, \nu_t) \big) \drm t\bigg].
\end{align*}
By definition of the supremum, for all $\varepsilon > 0$ and $(t,x,y,z,S) \in [0,T] \times \Omega \times \R \times \R^d \times \S_+^{d}$, there exists $u^\varepsilon \in U_t^\circ(x,S)$ such that $\Hc^\circ_{\rm A}(t,x,y,z,S)- h^\circ_{\rm A} (t, x, y, z, u^\varepsilon) \leq \varepsilon$. Replacing in the previous inequality, and using in addition the assumption that $k_{\rm A}$ is bounded, we deduce that there exists a constant $C > 0$ such that $V^\circ_{\rm A}(\xi,\Sigma) \geq \E^{\widebar \P} [Y_0] - C \varepsilon$.
By arbitrariness of $\varepsilon > 0$, we conclude that $V^\circ_{\rm A}(\xi,\Sigma) \geq \E^{\widebar \P}[Y_0]$. To summarise, for arbitrary $\nu \in \Uc^\circ(\Sigma)$, we have:
\begin{align*}
    \E^{\widebar \P}[Y_0] = V^\circ_{\rm A} \big(\xi,\Sigma\big) \geq J_{\rm A} \big(\xi,\P^{\nu} \big),
\end{align*}
and, since $\xi$ is admissible, the participation constraint must be satisfied, which implies $\E^{\widebar \P}[Y_0] \geq R_{\rm A}$. This already concludes the proof of $(i)$. Finally, recall that we implicitly assume in \Cref{pb:PA_FB_reformulation} that for any $\Sigma \in \Sc$ and $\xi \in \Xi^\circ(\Sigma)$, there should exist an optimal response for the agent, namely $\P^\circ \in \Pc^\circ(\Sigma)$ such that $V^\circ_{\rm A} (\xi,\Sigma) = J_{\rm A} (\xi, \P^\circ)$. Therefore, looking back at \eqref{eq:reward}, we deduce from the existence of $\P^\circ \in \Pc^{\circ,\star}(\xi,\Sigma)$ that
\begin{align*}
    \E^{\P^\circ} \bigg[\int_0^T \Kc^{\P^\circ}_{\rm A}(t) \Big( \Hc^\circ_{\rm A}(t,X,Y_t,Z_t,\Sigma_t) - h^\circ_{\rm A} \big(t, X, Y_t, Z_t, \nu^{\P^\circ}_t \big) \Big) \drm t \bigg] = 0,
\end{align*}
which implies, since the discount factor is (strictly) positive, that
\begin{align*}
    \Hc^\circ_{\rm A} \big(t, X, Y_t, Z_t, \Sigma_t \big) = h^\circ_{\rm A} \big(t, X, Y_t, Z_t, \nu_t^{\P^\circ} \big), \quad \drm t \otimes \P^\circ\text{--a.e. on } [0,T] \times \Omega.
\end{align*}
In particular, this completes the proof of $(ii)$, and we can further deduce that $Z$ satisfies \eqref{eq:optimal_effort_FB} in \Cref{def:contract_FB}.
\end{proof}

\Cref{prop:solve_agent_pb} provides a BSDE representation for the agent's dynamic value function, given a quadratic variation density process $\Sigma \in \Sc$ and a contract $\xi \in \Xi^\circ(\Sigma)$. This already induces a first natural representation for any contract $\xi \in \Xi^\circ(\Sigma)$, namely as the terminal value of an appropriate process $Y \in \D^p$, defined as the well-posed solution to ODE \eqref{eq:ODE_FB}. The main difference with the desired form of contracts \eqref{eq:contract_FB} lies in the initial value, which is a random variable $Y_0$ in \eqref{eq:ODE_FB} but a constant $y_0$ in \eqref{eq:contract_FB}. The following result ultimately proves that the restriction to contracts of the form $\xi = Y_T^{y_0,Z}$ as in \eqref{eq:contract_FB}, for some process $Z \in \Vc^\circ(\Sigma)$ and a constant $y_0 \geq R_{\rm A}$, is without loss of generality.

\begin{theorem}\label{thm:solution_FB}
We have $V_{\rm P}^\circ = \widetilde V^\circ_{\rm P}$, where $\widetilde V^\circ_{\rm P}$ is defined as follows,
\begin{align}\label{pb_principal_FB_simple}
    \widetilde V^\circ_{\rm P} := \sup_{y_0 \geq R_{\rm A}} \underline V_{\rm P}^\circ (y_0), \; \text{ with } \;
    \underline V_{\rm P}^\circ (y_0) := \sup_{\Sigma \in \Sc} \sup_{Z \in \Vc^\circ(\Sigma)} \sup_{\P^\circ \in \Pc^{\circ,\star}(Y_T^{y_0,Z},\Sigma)} J_{\rm P} \big(Y_T^{y_0,Z}, \P^\circ \big),
\end{align}
\modifreview{where} for $y_0 \geq R_{\rm A}$, $\Sigma \in \Sc$, and $Z \in \Vc^\circ(\Sigma)$,
\begin{enumerate}[label=$(\roman*)$]
    \item $Y^{y_0,Z}$ is defined as the solution to {\rm ODE} \eqref{eq:contract_FB};
    \item $\P^\circ \in \Pc^{\circ,\star} \big(Y_T^{y_0,Z},\Sigma \big)$ if and only if \eqref{eq:optimal_effort_FB} holds for $\P^\circ$;
    \item $V_{\rm A}^\circ \big(Y^{y_0,Z},\Sigma \big) = y_0 \geq R_{\rm A}$.
\end{enumerate}
\end{theorem}

While it is expected that $V_{\rm P}^\circ \geq \widetilde V^\circ_{\rm P}$, since the principal's problem corresponding to the value $\widetilde V^\circ_{\rm P}$ is by definition restricted to contracts of the form \eqref{eq:contract_FB}, the previous result highlights that we actually have equality. In other words, to solve the principal's problem in the `first-best' reformulation stated in \Cref{pb:PA_FB_reformulation}, one can restrict the study to contracts of the form \eqref{eq:contract_FB}, without loss of generality. Once again, we wish to insist on the fact that the proof of the previous theorem, detailed below, mostly follows the arguments already developed in the proof of \Cref{prop:solve_agent_pb}, and therefore does not rely on the theory of 2BSDEs.

\begin{proof}[\Cref{thm:solution_FB}]
\emph{Step 1.} We first argue that the inequality $V_{\rm P}^\circ \geq \widetilde V^\circ_{\rm P}$ holds. As already mentioned, we clearly have that, for fixed $\Sigma \in \Sc$, any contract defined by $\xi = Y_T^{y_0,Z}$ through \eqref{eq:contract_FB} for some $y_0 \in \R$ and $Z \in \Vc^\circ(\Sigma)$ is an $\Fc_T$-measurable random variable, satisfying in addition the integrability condition \eqref{eq:integrability_contract_agent} by definition of $Z \in \Vc^\circ$. To conclude that $\xi \in \Xi^\circ(\Sigma)$, it suffices to show that this contract $\xi$ satisfies in addition the agent's participation constraint, namely $(iii)$ in the theorem. To prove this, we can follow the same arguments as in Step 2 in the proof of \Cref{prop:solve_agent_pb}. More precisely, through similar computations, we obtain
\begin{align*}
    V^\circ_{\rm A} \big(Y_T^{y_0,Z},\Sigma \big) &= \sup_{\P \in \Pc^\circ(\Sigma)} \E^{\P}\bigg[ y_0 - \int_0^T \Kc^\P_{\rm A}(t) \Big( \Hc^\circ_{\rm A} \big(t,X,Y_t^{y_0,Z},Z_t,\Sigma_t \big)-h^\circ_{\rm A} \big(t, X, Y_t^{y_0,Z}, Z_t, \nu^\P_t \big) \Big) \drm t\bigg]\\
    &= y_0 - \inf_{\P \in \Pc^\circ(\Sigma)} \E^{\P} \bigg[\int_0^T \Kc^\P_{\rm A}(t) \Big(\Hc^\circ_{\rm A}\big(t,X,Y_t^{y_0,Z},Z_t,\Sigma_t \big) - h^\circ_{\rm A} \big(t, X, Y_t^{y_0,Z}, Z_t, \nu^\P_t \big) \Big) \drm t \bigg].
\end{align*}
First, by definition of $\Hc_{\rm A}^\circ$ and $h_{\rm A}^\circ$, and recalling that the discount factor $\Kc^\P_{\rm A}$ is (strictly) positive, we clearly have $V^\circ_{\rm A} \big(Y_T^{y_0,Z},\Sigma \big) \leq y_0$. Then, since $Z \in \Vc^\circ(\Sigma)$, we have, by \Cref{def:contract_FB} $(ii)$ that there exists $\P^\circ \in \Pc^\circ(\Sigma)$ such that \eqref{eq:optimal_effort_FB} holds, namely
\begin{align*}
    \Hc^\circ_{\rm A} \big(t, X, Y_t^{y_0,Z}, Z_t, \Sigma_t \big) = h^\circ_{\rm A} \big(t, X, Y_t^{y_0,Z}, Z_t, \nu_t^{\P^\circ} \big), \quad \drm t \otimes \P^\circ\text{--a.e. on } [0,T] \times \Omega.
\end{align*}
We thus deduce the equality $V^\circ_{\rm A} \big(Y_T^{y_0,Z},\Sigma \big) = y_0$ for such $\P^\circ$, with associated effort $\nu^{\P^\circ}$ defined as a maximiser of the Hamiltonian $\Hc_{\rm A}^\circ$. Finally, since $y_0 \geq R_{\rm A}$, we conclude that $Y_T^{y_0,Z} \in \Xi^\circ(\Sigma)$.

\medskip

\emph{Step 2.} We now prove the converse inequality, namely $V_{\rm P}^\circ \leq \widetilde V^\circ_{\rm P}$. Recall that, by \Cref{prop:solve_agent_pb}, any contract $\xi \in \Xi^\circ(\Sigma)$ can be written as $\xi = Y_T$, where
\begin{align*}
    Y_t = Y_0 - \int_0^t \Hc_{\rm A}^\circ \big(s, X, Y_s, Z_s, \Sigma_s \big) \drm s + \int_0^t Z_s \cdot \drm X_s, \; t \in [0,T], \; \widebar \P\textnormal{--a.s.},
\end{align*}
with $Y_0$ an $\Fc_0^{\widebar \P+}$-measurable random variable and $Z \in \H^p$. Moreover, we have $V^\circ_{\rm A}(\xi,\Sigma) = \E^{\widebar \P}[Y_0] \geq R_{\rm A}$, and $\P^\circ \in \Pc^{\circ,\star}(\xi,\Sigma)$ if and only if $\Hc^\circ_{\rm A} (t, X, Y_t, Z_t, \Sigma_t) = h^\circ_{\rm A} (t, X, Y_t, Z_t, \nu_t^{\P^\circ})$, $\drm t \otimes \P^\circ\text{--a.e. on } [0,T] \times \Omega$. To obtain the desired contract form \eqref{eq:contract_FB}, we need to show that, without loss of generality, the random variable $Y_0$ can be replaced by a constant $y_0 \geq R_{\rm A}$, and verify that we can choose $Z \in \Vc^\circ(\Sigma)$. With this in mind, let $y_0:=\E^{\widebar \P}[Y_0] \geq R_{\rm A}$ and $Z \in \H^p$, to consider the contract defined by $\xi = Y_T^{y_0,Z}$ through \eqref{eq:contract_FB}. By Step 1, we already have that 
\begin{align}\label{eq:value_proof}
    V^\circ_{\rm A} \big(Y_T^{y_0,Z},\Sigma \big) &= y_0 - \inf_{\P \in \Pc^\circ(\Sigma)} \E^{\P} \bigg[\int_0^T \Kc^\P_{\rm A}(t) \Big(\Hc^\circ_{\rm A}\big(t,X,Y_t^{y_0,Z},Z_t,\Sigma_t \big) - h^\circ_{\rm A} \big(t, X, Y_t^{y_0,Z}, Z_t, \nu^\P_t \big) \Big) \drm t \bigg],
\end{align}
and, in particular, $V^\circ_{\rm A} \big(Y_T^{y_0,Z},\Sigma \big) \leq y_0$. Using the same arguments as in Step 2 in the proof of \Cref{prop:solve_agent_pb}, we can also prove the converse inequality, and thus conclude that $V^\circ_{\rm A} \big(Y_T^{y_0,Z},\Sigma \big) = y_0$. Finally, recalling again that $\Pc^{\circ,\star}(\xi,\Sigma) \neq \varnothing$, we deduce that $\P^\circ \in \Pc^{\circ,\star}(\xi,\Sigma)$ if and only if
\begin{align*}
    \E^{\P^\circ} \bigg[\int_0^T \Kc^{\P^\circ}_{\rm A}(t) \Big(\Hc^\circ_{\rm A}\big(t,X,Y_t^{y_0,Z},Z_t,\Sigma_t \big) - h^\circ_{\rm A} \big(t, X, Y_t^{y_0,Z}, Z_t, \nu^{\P^\circ}_t \big) \Big) \drm t \bigg] = 0,
\end{align*}
which directly implies that \eqref{eq:optimal_effort_FB} holds for $\P^\circ$, or equivalently that $Z$ satisfies \eqref{eq:optimal_effort_FB} in \Cref{def:contract_FB}. To summarise, for a contract $Y_T^{y_0,Z} \in \Xi^\circ(\Sigma)$, we have $V^\circ_{\rm A} \big(Y_T^{y_0,Z},\Sigma \big) = y_0$, and this value is achieved for the optimal effort $\nu^{\P^\circ}$, defined as a maximiser of the Hamiltonian $\Hc_{\rm A}^\circ$. We conclude that both contracts $Y_T$ and $Y_T^{y_0,Z}$ generate the same efforts, and induce the same value for the agent. Therefore, without loss of generality, the principal's optimisation problem can be restricted to contracts of the form \eqref{eq:contract_FB}, namely $\xi := Y_T^{y_0,Z}$, where the process $Y^{y_0,Z}$ is defined as the solution to ODE \eqref{eq:contract_FB}, for $y_0 \geq R_{\rm A}$. It remains to verify that $Z \in \Vc^\circ(\Sigma)$, to conclude that $\widetilde V_{\rm P}^\circ = V_{\rm P}^\circ$. Note that $Z \in \H^p$ and satisfies \Cref{def:contract_FB} $(ii)$, as mentioned above. Finally, following similar arguments as in Step 1 in the proof of \Cref{prop:solve_agent_pb}, one can prove that $Y^{y_0,Z} \in \D^p$, so that \Cref{def:contract_FB} $(i)$ is also satisfied. 
\end{proof}

To summarise, \Cref{thm:solution_FB} allows to conclude that, when the quadratic variation is fixed by the principal through its density $\Sigma \in \Sc$, the relevant form of contract is given by $\xi := Y_T^{y_0, Z}$ where the process $Y^{y_0, Z}$ is defined by \eqref{eq:contract_FB}, for some parameters $y_0 \geq R_{\rm A}$ and $Z \in \Vc^\circ(\Sigma)$ optimally chosen by the principal. Moreover, since the agent's optimal response $\P^\circ \in \Pc^{\circ,\star}(Y^{y_0,Z},\Sigma)$ satisfies \eqref{eq:optimal_effort_FB}, his optimal effort at time $t \in [0,T]$ can be represented by $\nu^\circ_t := u^\circ (t, X, Y_t^{y_0,Z}, Z_t, \Sigma_t)$, where the function $u^\circ$ coincides with a maximiser of the Hamiltonian $\Hc_{\rm A}^\circ$, \textit{i.e.} 
\begin{align}\label{eq:u*_FB}
	u^\circ (t,x,y,z,S) \in \argmax_{u \in U_t^\circ(x,S)} h^\circ_{\rm A} (t, x, y, z, u), \quad (t,x,y,z,S) \in [0,T] \times \Omega \times \R \times \R^d \times \S_+^{d}.
\end{align}
Finally, one can notice that, for fixed $y_0 \geq R_{\rm A}$, $\underline V_{\rm P}^\circ (y_0)$ corresponds to the value of a (more) standard stochastic control problem, with two state variables $X$ and $Y^{y_0,Z}$, controlled through $\Sigma \in \Sc$ and $Z \in \Vc^\circ(\Sigma)$. One can finally compute the dynamics of these state variables under $\P^\circ \in \Pc^{\circ,\star}(Y^{y_0,Z},\Sigma)$, and conclude that 
solving \Cref{pb:PA_FB_reformulation} is equivalent to solving \Cref{pb:first-best_solution}, introduced below.

\begin{problem}\label{pb:first-best_solution}
Consider the (more) standard stochastic control problem \eqref{pb_principal_FB_simple}, namely
\begin{align*}
	\widetilde V^\circ_{\rm P} := \sup_{y_0 \geq R_{\rm A}} \underline V_{\rm P}^\circ (y_0), \; \text{ with } \;
	\underline V_{\rm P}^\circ (y_0) := \sup_{\Sigma \in \Sc} \sup_{Z \in \Vc^\circ(\Sigma)} \sup_{\P^\circ \in \Pc^{\circ,\star}(Y_T,\Sigma)} J_{\rm P} \big(\P^\circ, Y_T \big).
\end{align*}
where the couple of state variables $(X,Y)$ is the solution to the following system of {\rm SDEs:}
\begin{subequations}\label{eq:dynamics-FB}
\begin{align}
	\drm X_t &= \sigma \big(t, X, \nu^{\P^\circ}_t \big) \Big( \lambda \big(t, X, \nu^{\P^\circ}_t \big) \mathrm{d}t + \mathrm{d}W_t \Big), \qquad \qquad \qquad \qquad \quad \; \;  t\in[0,T], \\
    \drm Y_t &= \Big( c_{\rm A} \big(t, X, \nu^{\P^\circ}_t \big) + Y_t \, \modifreview{k_{\rm A}} \big(t, X, \nu^{\P^\circ}_t \big) \Big)  \drm t
	+ Z_t \cdot \sigma \big(t, X, \nu^{\P^\circ}_t \big) \mathrm{d}W_t, \quad t\in[0,T],
\end{align}
\end{subequations}
coupled through the agent's optimal effort $\nu^{\P^\circ}_t := u^\circ (t, X, Y_t, Z_t, \Sigma_t)$ under $\P^\circ \in \Pc^{\circ,\star}(Y_T,\Sigma)$.
\end{problem}

One can already notice that \Cref{pb:first-best_solution} stated above is extremely similar to \Cref{pb:PA_solution_CPT}. Indeed, in both problems, the principal optimises the same criteria, and, apart from the agent's optimal effort, the dynamics of the state variables in both problems are identical.
The only yet main difference between the two problems is that, in \Cref{pb:first-best_solution}, this optimal effort depends (among other variables) on $\Sigma$, while, in \Cref{pb:PA_solution_CPT}, it is instead a function of $\Gamma$. We will nevertheless prove, in the next section, that there is a one-to-one correspondence between $\Sigma$ and $\Gamma$ so that optimising one or the other actually lead to the same value for the principal, \textit{i.e.} $\widetilde V_{\rm P} = \widetilde V^\circ_{\rm P}$.

\subsection{Penalisation contracts}\label{ss:penalisation_contract}

As mentioned in introduction, principal--agent problems with moral hazard are equivalent to their first-best counterpart if the principal actually observes the agent's efforts \emph{and} can design a \emph{forcing} contract, which basically forces him to implement the recommended efforts. We wish to follow the same type of reasoning here, to conclude on the equivalence between the original problem and its `first-best' reformulation, introduced by \Cref{pb:PA_FB_reformulation}. Notice that we already proved in \Cref{lem:pb3>pb1} that the principal's value $V_{\rm P}^\circ$ in the alternative `first-best' problem (\Cref{pb:PA_FB_reformulation}) is greater than the value $V_{\rm P}$ in the original problem (\Cref{pb:PA_original}). Therefore, it remains to introduce appropriate contracts to show that, starting from \Cref{pb:PA_original}, one can actually achieve the `first-best' value $V_{\rm P}^\circ$. This will imply the converse inequality, namely $V_{\rm P} \geq V_{\rm P}^\circ$, and thus conclude the proof of the main result of this paper, stated below.

\begin{theorem}\label{thm:main}
	\modif{Under {\rm \Cref{ass:weak_uniqueness,ass:duality}}, solving {\rm \Cref{pb:PA_FB_reformulation}} is equivalent to solving {\rm \Cref{pb:PA_original}}, \textit{i.e.} $V_{\rm P}^\circ = V_{\rm P}$.}
\end{theorem}

The idea is to start from the original problem \eqref{eq:pb_principal} and restrict the study to contracts of the form \eqref{eq:contract_CPT} introduced in \cite{cvitanic2018dynamic}, hence considering \Cref{pb:PA_solution_CPT}. Then, we only have to prove that this problem is equivalent to \Cref{pb:first-best_solution}. Indeed, we already have by \Cref{lem:pb1>pb2} that $V_{\rm P} \geq \widetilde V_{\rm P}$, and by \Cref{thm:solution_FB} that $V^\circ_{\rm P} = \widetilde V^\circ_{\rm P}$. Therefore, it only remains to prove that $\widetilde V_{\rm P} = \widetilde V^\circ_{\rm P}$ to achieve the desired result, namely $V_{\rm P} \geq \widetilde V_{\rm P} = \widetilde V^\circ_{\rm P} = V^\circ_{\rm P}$. \modif{The previous equality is proved in two steps, first showing that $\widetilde V_{\rm P} \geq \widetilde V^\circ_{\rm P}$ in \Cref{lem:pb2<=pb4}, and then the converse inequality in \Cref{lem:pb2>=pb4}. Indeed, while the first inequality is quite straightforward to prove, the second inequality seems to involve an additional assumption on the structure of the `constrained' Hamiltonian $\Hc_{\rm A}^\circ$, and formulated in \Cref{ass:duality} stated below. While this assumption does not appear explicitly in \cite{cvitanic2018dynamic}, the proof of their main result, namely Theorem 3.6 on the optimality of the contract form \eqref{eq:contract_CPT}, may actually fail if this assumption is not verified. We refer to \Cref{rk:duality} below for further clarification.}

\medskip

\modif{
\begin{lemma}\label{lem:pb2<=pb4}
Let $y_0 \geq R_{\rm A}$, $(Z,\Gamma) \in \Vc$, and consider the corresponding state variables $(X,Y)$ with dynamics given by \eqref{eq:dynamics-CPT} and an associated optimal effort $\nu^\star_t := u^\star (t, X, Y_t, Z_t, \Gamma_t)$, $t\in [0,T]$. Then there exists a process $\Sigma \in \Sc$ defined by
\begin{align*}
	\Sigma_t := \big[ \sigma \sigma^\top \big] \big( t, X, u^\star (t, X, Y_t, Z_t, \Gamma_t) \big), \quad \drm t \otimes \P^\star\text{-a.e. on } [0,T]\times \Omega,
\end{align*}
such that $Z \in \Vc^\circ(\Sigma)$ and $\nu^\circ_t := u^\circ (t, X, Y_t, Z_t, \Sigma_t) = u^\star (t, X, Y_t, Z_t, \Gamma_t)$, $\drm t \otimes \P^\star\text{-a.e. on } [0,T]\times \Omega$. 
\end{lemma}}

\modif{
\begin{proof}[\Cref{lem:pb2<=pb4}]
Fix $(t,x,y,z,\gamma) \in [0,T] \times \Omega \times \R \times \R^d \times \M^{d}$. First, by definition of the function $u^\star$ in \eqref{eq:u*_CPT},
\begin{align*}
	\Hc_{\rm A} (t, x, y, z, \gamma) = h_{\rm A} \big(t, x, y, z, \gamma, u^\star (t,x,y,z,\gamma) \big).
\end{align*}
Alternatively, by \Cref{rk:hamiltonian}, we have 
\begin{align*}
	\Hc_{\rm A} (t, x, y, z, \gamma) &= \sup_{S \in \S_t(x) } \bigg\{ \Hc^\circ_{\rm A} (t, x, y, z, S) + \dfrac12 {\rm Tr} \big[ \gamma S \big] \bigg\},
\end{align*}
recalling that $\S_t(x)$ is defined for all $(t,x) \in [0,T] \times \Omega$ by \eqref{eq:set_Sigma}. Define $S^\star := \big[ \sigma \sigma^\top \big] \big( t, x, u^\star (t, x, y, z, \gamma) \big)$. By definition, $S^\star \in \S_t(x)$, and we thus have
\begin{align*}
	\Hc_{\rm A} (t, x, y, z, \gamma) &\geq \Hc^\circ_{\rm A} \big(t, x, y, z, S^\star \big) + \dfrac12 {\rm Tr} \big[ \gamma \big[ \sigma \sigma^\top \big] \big( t, x, u^\star (t, x, y, z, \gamma) \big) \big].
\end{align*}
Moreover, also by definition of $S^\star$, we have $u^\star (t, x, y, z, \gamma) \in U_t^\circ(x,S^\star)$, implying
\begin{align*}
	\Hc^\circ_{\rm A} \big(t, x, y, z, S^\star \big) := \sup_{u \in U_t^\circ(x,S^\star)} \; h^\circ_{\rm A} (t, x, y, z, u)
	\geq h^\circ_{\rm A} \big( t, x, y, z, u^\star(t,x,y,z,\gamma) \big).
\end{align*}
Combining all the previous equations we obtain the following inequalities,
\begin{align*}
	h_{\rm A} \big(t, x, y, z, \gamma, u^\star (t,x,y,z,\gamma) \big) 
	&\geq \Hc^\circ_{\rm A} \big(t, x, y, z, S^\star \big) + \dfrac12 {\rm Tr} \big[ \gamma \big[ \sigma \sigma^\top \big] \big( t, x, u^\star (t, x, y, z, \gamma) \big) \big] \\
	&\geq h^\circ_{\rm A} \big( t, x, y, z, u^\star(t,x,y,z,\gamma) \big) + \dfrac12 {\rm Tr} \big[ \gamma \big[ \sigma \sigma^\top \big] \big( t, x, u^\star (t, x, y, z, \gamma) \big) \big] \\
	&= h_{\rm A} \big( t, x, y, z, \gamma, u^\star(t,x,y,z,\gamma) \big),
\end{align*}
which necessarily implies equality at every lines, and in particular
\begin{align*}
	\Hc^\circ_{\rm A} \big(t, x, y, z, S^\star \big) = h^\circ_{\rm A} \big( t, x, y, z, u^\star(t,x,y,z,\gamma) \big).
\end{align*}
Therefore, the map $u^\star$ is a maximiser of the constrained Hamiltonian $\Hc^\circ_{\rm A}$ for the quadratic variation $S^\star$. 

\medskip

To summarise, there exists a map $S^\star$ such that for all $(t,x,y,z,\gamma) \in [0,T] \times \Omega \times \R \times \R^d \times \M^{d}$, we have $S^\star(t,x,y,z,\gamma) \in \S_t(x)$ and
\begin{align*}
	u^\circ \big( t,x,y,z, S^\star(t,x,y,z,\gamma) \big) = u^\star (t,x,y,z,\gamma).
\end{align*}
In other words, for any $(Z,\Gamma) \in \Vc$ and the corresponding state variables $(X,Y)$ with dynamics given by \eqref{eq:dynamics-CPT}, one can define a process $\Sigma \in \Sc$ by taking
\begin{align*}
	\Sigma_t := S^\star (t,X,Y_t,Z_t,\Gamma_t) = \big[ \sigma \sigma^\top \big] \big( t, X, u^\star (t, X, Y_t, Z_t, \Gamma_t) \big), \quad \drm t \otimes \P^\star\text{-a.e. on } [0,T]\times \Omega,
\end{align*}
It is then straightforward to verify that $Z \in \Vc^\circ(\Sigma)$ and, by the previous reasoning,
\begin{align*}
	\nu^\circ_t := u^\circ (t,X,Y_t,Z_t,\Sigma_t) = u^\star (t,X,Y_t,Z_t,\Gamma_t), \quad \drm t \otimes \P^\star\text{-a.e. on } [0,T]\times \Omega.
\end{align*}
\end{proof}}

\modif{The previous lemma highlights that, for any $y_0 \geq R_{\rm A}$ and $(Z,\Gamma) \in \Vc$, one can find an appropriate process $\Sigma \in \Sc$ such that $Z \in \Vc^\circ(\Sigma)$, and the optimal efforts $\nu^\star$ and $\nu^\circ$ coincide. Given the formulation of \Cref{pb:PA_solution_CPT} and \Cref{pb:first-best_solution}, this is enough to conclude that $\widetilde V_{\rm P} \geq \widetilde V^\circ_{\rm P}$. We now turn to the converse inequality, for which we require the following assumption.}


\modif{
\begin{assumption}\label{ass:duality}
	For all $(t,x,y,z) \in [0,T] \times \Omega \times \R \times \R^d$ and $S \in \S_t(x)$, we have
    \begin{align}\label{eq:duality}
		\Hc^\circ_{\rm A} (t, x, y, z, S) = \inf_{\gamma \in \M^d} \bigg\{ \Hc_{\rm A} (t, x, y, z, \gamma) - \dfrac12 {\rm Tr} \big[ \gamma S \big] \bigg\}.
	\end{align}
    Moreover, the infimum is achieved, namely there exists a measurable function $\gamma^\star : [0,T] \times \Omega \times \R \times \R^d \times \S_d^+ \longrightarrow \M^d$ such that
    \begin{align}\label{eq:inf_achieved}
        \Hc^\circ_{\rm A} (t, x, y, z, S) = \Hc_{\rm A} \big(t, x, y, z, \gamma^\star(t,x,y,z,S) \big) - \dfrac12 {\rm Tr} \big[ \gamma^\star(t,x,y,z,S) S \big].
    \end{align}
\end{assumption}}

\modif{
\begin{remark}\label{rk:duality}
	While the previous assumption is not explicitly stated in {\rm \cite{cvitanic2018dynamic}}, the proof of the main result regarding the optimality of the contract form \eqref{eq:contract_CPT} may fail if this assumption is not satisfied. More precisely, it is claimed in this proof that the function defined in {\rm (4.17)}, namely using our notations
	\begin{align*}
		\widebar f : \gamma \in \M^d \longmapsto \Hc_{\rm A} (t, x, y, z, \gamma) - \Hc^\circ_{\rm A} (t, x, y, z, S) - \dfrac12 {\rm Tr} \big[ \gamma S \big],
		\quad (t,x,y,z,S) \in [0,T] \times \Omega \times \R \times \R^d \times \S_d^+,
	\end{align*}
	is surjective on $(0,\infty)$. However, for this claim to be true, one need at least
	\begin{align*}
		\inf_{\gamma \in \M^d} \widebar f(\gamma) = 0,
	\end{align*}
    and this previous condition can be rewritten as \eqref{eq:duality}, for all $(t,x,y,z,S) \in [0,T] \times \Omega \times \R \times \R^d \times \S_d^+$. In other words, $(- 2 \, \Hc^\circ_{\rm A})$ should be the convex conjugate of $2 \, \Hc_{\rm A}$. However, as already mentioned in {\rm \Cref{rk:hamiltonian}}, $2 \, \Hc_{\rm A}$ is already the convex conjugate of $(- 2 \, \Hc^\circ_{\rm A})$. Therefore, the previous condition is true if and only if $(- 2 \, \Hc^\circ_{\rm A})$ actually coincide with its \emph{biconjugate} (the convex conjugate of the convex conjugate). Note that, for any function $g$, its biconjugate $g^{\star \star}$ satisfies $g^{\star \star} \leq g$, which implies here $(- 2 \, \Hc^\circ_{\rm A})^{\star \star} \leq - 2 \, \Hc^\circ_{\rm A}$, and thus
	\begin{align*}
		\Hc^\circ_{\rm A} (t, x, y, z, S) \leq \inf_{\gamma \in \M^d} \bigg\{ \Hc_{\rm A} (t, x, y, z, \gamma) - \dfrac12 {\rm Tr} \big[ \gamma S \big] \bigg\}.
	\end{align*}
	However, by the Fenchel--Moreau theorem, the inequality $g^{\star \star} \leq g$ becomes an equality \emph{if and only if} $g$ is convex, lower semi-continuous and proper. Since $\Hc^\circ_{\rm A}$, as a function of $S$, does not always satisfy these conditions in the framework under consideration here and in {\rm \cite{cvitanic2018dynamic}}, this naturally leads to the first part of {\rm \Cref{ass:duality}}. The second part of the assumption, namely that the infimum is achieved, is also needed in the proof of {\rm \cite[Theorem 3.6]{cvitanic2018dynamic}}. More precisely, from the fact that $\widebar f$ is surjective on $(0,\infty)$, it is deduced that for any $\dot K \geq 0$, there exists a measurable process $\Gamma$ such that 
    \begin{align*}
        \dot K_t = \Hc_{\rm A} (t, X, Y_t, Z_t, \Gamma_t) - \Hc^\circ_{\rm A} (t, X, Y_t, Z_t, \Sigma_t) - \dfrac12 {\rm Tr} \big[ \Gamma_t \Sigma_t \big].
    \end{align*}
    However, as mentioned, it is possible that $\dot K_t = 0$ for some $t \in [0,T]$, and thus one need the function $\widebar f$ to actually be surjective onto $[0,\infty)$ instead of $(0,+\infty)$ only. This naturally leads to \eqref{eq:inf_achieved} in the second part of {\rm \Cref{ass:duality}}.\footnote{\modif{We have been in contact with the authors of \cite{cvitanic2018dynamic} to verify the relevance of our additional assumption. We are currently discussing potential ways to construct alternative forms of contracts in settings where \Cref{ass:duality} does not hold.}}
\end{remark}}

\modif{Under \Cref{ass:duality}, we can state the following result, which serves as the counterpart to \Cref{lem:pb2<=pb4}. Its proof follows similar arguments, adapted to establish the reverse inequality, and is reported below for completeness.
\begin{lemma}\label{lem:pb2>=pb4}
	Let $y_0 \geq R_{\rm A}$, $\Sigma \in \Sc$, $Z \in \Vc^\circ(\Sigma)$, and consider the corresponding state variables $(X,Y)$ with dynamics given by \eqref{eq:dynamics-FB}, as well as an optimal effort $\nu^\circ_t := u^\circ (t, X, Y_t, Z_t, \Sigma_t)$, $t\in [0,T]$. Under {\rm \Cref{ass:duality}}, there exists a process $\Gamma \in \M^{d}$ defined by
	\begin{align*}
		\Gamma_t \in \argmin_{\gamma \in \M^{d}} \bigg\{ \Hc_{\rm A} (t, X, Y_t, Z_t, \gamma) - \dfrac12 {\rm Tr} \big[ \gamma \Sigma_t \big] \bigg\}, \quad \drm t \otimes \P^\circ\text{-a.e. on } [0,T]\times \Omega,
	\end{align*}
	such that $(Z,\Gamma) \in \Vc$ and $\nu^\star_t := u^\star (t, X, Y_t, Z_t, \Gamma_t) = u^\circ (t, X, Y_t, Z_t, \Sigma_t)$, $\drm t \otimes \P^\circ\text{-a.e. on } [0,T]\times \Omega$.
\end{lemma}}

\modif{
\begin{proof}[\Cref{lem:pb2>=pb4}]
	Fix $(t,x,y,z) \in [0,T] \times \Omega \times \R \times \R^d$ and $S \in \S_t(x)$. On the one hand, by definition of the function $u^\circ$ in \eqref{eq:u*_FB}, we have
	\begin{align*}
		\Hc^\circ_{\rm A} (t, x, y, z, S) = h_{\rm A}^\circ \big(t, x, y, z, u^\circ (t,x,y,z,S) \big).
	\end{align*}
	On the other hand, by \Cref{ass:duality}, there exists a measurable function $\gamma^\star : [0,T] \times \Omega \times \R \times \R^d \times \S_d^+ \longrightarrow \M^d$ such that \eqref{eq:inf_achieved} holds, namely
    \begin{align*}
        \Hc^\circ_{\rm A} (t, x, y, z, S) = \Hc_{\rm A} \big(t, x, y, z, \gamma^\star(t,x,y,z,S) \big) - \dfrac12 {\rm Tr} \big[ \gamma^\star(t,x,y,z,S) S \big].
    \end{align*}
    Then, by definition of $\Hc_{\rm A}$, we have for any $\gamma \in \M^d$
	\begin{align}\label{eq:optimal}
		\Hc_{\rm A} (t, x, y, z, \gamma) \geq h_{\rm A} \big(t, x, y, z, \gamma, u^\circ (t,x,y,z,S) \big),
	\end{align}
	and since $[\sigma \sigma^\top ](t,x,u^\circ (t,x,y,z,S)) = S$, we deduce
	\begin{align*}
		\Hc_{\rm A} (t, x, y, z, \gamma) - \dfrac12 {\rm Tr} \big[ \gamma S \big]
		&\geq h_{\rm A} \big(t, x, y, z, \gamma, u^\circ (t,x,y,z,S) \big) - \dfrac12 {\rm Tr} \big[ \gamma [\sigma \sigma^\top ](t,x,u^\circ (t,x,y,z,S)) \big] \\
		&= h_{\rm A}^\circ \big(t, x, y, z, u^\circ (t,x,y,z,S) \big). \nonumber
	\end{align*}
	Combining all the previous equalities and inequalities together for $\gamma:= \gamma^\star(t,x,y,z,S)$, we have
    \begin{align*}
        h_{\rm A}^\circ \big(t, x, y, z, u^\circ (t,x,y,z,S) \big) &= \Hc^\circ_{\rm A} (t, x, y, z, S)
		\geq h_{\rm A}^\circ \big(t, x, y, z, u^\circ (t,x,y,z,S) \big),
    \end{align*}
    thus implying equality everywhere, in particular in \eqref{eq:optimal} for $\gamma:= \gamma^\star(t,x,y,z,S)$. 
	Therefore, the map $u^\circ$ is a maximiser of the Hamiltonian $\Hc_{\rm A}$ when choosing $\gamma^\star$. 
	
	\medskip
	
	To summarise, there exists a map $\gamma^\star$ such that for all $(t,x,y,z) \in [0,T] \times \Omega \times \R \times \R^d$ and $S \in \S_t(x)$, we have
	\begin{align*}
		u^\star \big( t,x,y,z, \gamma^\star(t,x,y,z,S) \big) = u^\circ (t,x,y,z,S).
	\end{align*}
	In other words, for any $\Sigma \in \Sc$, $Z \in \Vc^\circ(\Sigma)$ and the corresponding state variables $(X,Y)$ with dynamics given by \eqref{eq:dynamics-FB}, one can define a process $\Gamma$ by taking
	\begin{align*}
		\Gamma_t = \gamma^\star(t,X,Y_t,Z_t,\Sigma_t), \quad \drm t \otimes \P^\star\text{-a.e. on } [0,T]\times \Omega
	\end{align*}
	It is then straightforward to verify that $(Z,\Gamma) \in \Vc$ and, by the previous reasoning,
	\begin{align*}
		\nu^\star_t := u^\star (t,X,Y_t,Z_t,\Gamma_t) = u^\circ (t,X,Y_t,Z_t,\Sigma_t), \quad \drm t \otimes \P^\star\text{-a.e. on } [0,T]\times \Omega.
	\end{align*}
\end{proof}}

\modif{Using \Cref{lem:pb2<=pb4} and \Cref{lem:pb2>=pb4}, we naturally deduce the following proposition.

\begin{proposition}\label{prop:pb2=pb4}
	Under {\rm \Cref{ass:duality}}, solving {\rm \Cref{pb:PA_solution_CPT}} is equivalent to solving {\rm \Cref{pb:first-best_solution}}, \textit{i.e.} $\widetilde V_{\rm P} = \widetilde V^\circ_{\rm P}$.
\end{proposition}}

To summarise, using sequentially \Cref{lem:pb1>pb2}, \Cref{prop:pb2=pb4}, \Cref{thm:solution_FB} and finally \Cref{lem:pb3>pb1}, we obtain the following relationship for the values of the various problems,
\begin{align*}
    V_{\rm P} \geq \widetilde V_{\rm P} = \widetilde V^\circ_{\rm P} = V^\circ_{\rm P} \geq V_{\rm P},
\end{align*}
which directly implies equality of all the values, and in particular concludes the proof of \Cref{thm:main}. In other words, under \Cref{ass:weak_uniqueness,ass:duality}, the four problems outlined are in fact equivalent. 
\modifreview{We recall here that \Cref{ass:weak_uniqueness} is a standing assumption, required to be able to derive a `simple' form of contracts in the `first-best' problem, as mentioned in \Cref{rk:MRP2}. This assumption is not required to develop the approach in \cite{cvitanic2018dynamic}.} \modif{On the other hand, although \Cref{ass:duality} is not stated explicitly in \cite{cvitanic2018dynamic}, it is required to establish the equivalence between the two problems and ensure optimality of the contracts derived in \cite{cvitanic2018dynamic}, as explained in \Cref{rk:duality}.
While this assumption is satisfied in most of the models studied in the literature, see for example \Cref{ss:example1,ss:example2}, it may fail in some frameworks. In particular, \Cref{ss:counter-example} presents a `counter-example', in which \Cref{ass:duality} is not satisfied.}


\section{Illustrative Examples}\label{sec:examples}

In this section, we first present two illustrative examples, in order to highlight the link between \Cref{pb:PA_solution_CPT} and \Cref{pb:first-best_solution}, in particular the one-to-one correspondence between the contract's sensitivity parameter $\Gamma$ in \Cref{pb:PA_solution_CPT} and the quadratic variation density process $\Sigma$ in \Cref{pb:first-best_solution}. \modif{Additionally, we present a third example, in which \Cref{ass:duality} is not satisfied, leading to a `gap' between the value of the original problem and its first-best counterpart. This `counter-example' raises additional questions on the optimality of the contracts introduced in \cite{cvitanic2018dynamic}.} The introduction and resolution of these \modif{three} examples is rather informal, we refer to the previous sections for a more rigorous formulation and the theoretical results. \modifreview{We highlight that in all three examples, the coefficients are trivially Lipschitz continuous and satisfy the usual linear growth condition, thus \Cref{ass:weak_uniqueness} is always satisfied (see for example \citeayn[Theorem 2.2 and 3.1]{touzi2013optimal}).}

\subsection{A first-best case}\label{ss:example1}

We first consider a very simple example, in which the agent only controls the volatility of a one-dimensional output process. More precisely, we let
\begin{align}\label{eq:X_dynamic_example}
    X_t := x_0 + \int_0^t \nu_s \drm W_s, \; t \in [0,T], \; x_0 \in \R,
\end{align}
where $W$ here is a one-dimensional Brownian Motion, and $\nu$ is the agent's control, assumed to take values in $(0,1]$. Given a contract $\xi$, the agent wishes to maximise the following expected utility,
\begin{align*}
    J_{\rm A} (\xi, \nu) := \E^{\P^\nu} \bigg[ U_{\rm A} \bigg( \xi - \dfrac12 \int_0^T \dfrac{1}{|\nu_t|^2} \drm t \bigg) \bigg], \; \text{ with } \; U_{\rm A} (x) = - \erm^{- \gamma_{\rm A} x}, \; x \in \R, \; \gamma_{\rm A} > 0.
\end{align*}
Similarly, the principal's expected utility to be maximised is defined by
\begin{align*}
    J_{\rm P} (\xi,\nu) := \E^{\P^\nu} \big[ U_{\rm P} \big(X_T - \xi - h \langle X \rangle_T \big) \big], \; \text{ with } \; U_{\rm P} (x) = - \erm^{- \gamma_{\rm P} x}, \; x \in \R, \; \gamma_{\rm P} > 0, \; h > 0.
\end{align*}
We first recall that, since the principal observes $X$ in continuous time, she can also compute the associated quadratic variation process. Here, such process follows the dynamics $\drm \langle X \rangle_t = \nu_t^2 \drm t$, $t \in [0,T]$. In other words, its density process $\Sigma$ satisfies $\Sigma_t = \nu_t^2$, $t \in [0,T]$, and since $\nu$ takes values in $(0,1]$, it is clear that, by observing $\Sigma$ in continuous-time, the principal can deduce the agent's effort $\nu_t = \sqrt{\Sigma_t}$, $t \in [0,T]$. This example therefore illustrates a first-best scenario, since from her observation, the principal can actually deduce the agent's effort.

\medskip

If one follows the approach in \cite{cvitanic2018dynamic}, slightly adapted for CARA utility function, the optimal form of contract is
\begin{align*}
    \xi = Y_T^{y_0,Z,\Gamma} = U_{\rm A}^{-1}(y_0) - \int_0^T \Hc_{\rm A}(\Gamma_t) \drm t + \int_0^T Z_t \drm X_t + \dfrac12 \int_0^T \big(\Gamma_t + \gamma_{\rm A} |Z_t|^2 \big) \drm \langle X \rangle_t,
\end{align*}
for $y_0 \in \R$, $(Z,\Gamma)$ an $\R^2$-valued process, and where the agent's Hamiltonian is defined by
\begin{align*}
    \Hc_{\rm A} (\gamma) := \dfrac12 \sup_{u \in (0,1]} \big\{ \gamma u^2 - u^{-2} \big\}, \; \gamma \in \R.
\end{align*}
The maximiser of the previous Hamiltonian is given by 
\begin{align*}
    u^\star (\gamma) :=
    \begin{cases}
        (-\gamma)^{-1/4} & \text{if } \gamma < -1, \\
        1 & \text{if } \gamma \geq -1,
    \end{cases}
\end{align*}
meaning that the agent's optimal response to the contract $\xi$ is given by $\nu^\star_t = u^\star(\Gamma_t)$, $t \in[0,T]$. Moreover, given such contract, one can easily compute that the agent's value boils down to $y_0$, and therefore the principal should choose $y_0 \geq R_{\rm A}$, if $R_{\rm A}$ represents the agent's reservation utility. Computing the dynamics of $X$ and $Y^{y_0,Z,\Gamma}$ under the optimal effort, one obtains
\begin{align*}
    \drm X_t &= u^\star(\Gamma_t) \drm W_t, \text{ and } \;
    \drm Y_t^{y_0,Z,\Gamma} = \dfrac12 \big( |u^\star(\Gamma_t)|^{-2} + \gamma_{\rm A} |Z_t|^2 |u^\star(\Gamma_t)|^2 \big) \drm t + Z_t u^\star(\Gamma_t) \drm W_t, \quad t \in [0,T].
\end{align*}
To reformulate the principal's problem in a more simple way, one can actually define a new state variable, namely $L_t := X_t - Y_t^{y_0,Z,\Gamma} - h \langle X \rangle_t$, $t \in [0,T]$, so that the principal's problem becomes
\begin{align}\label{ex:principal_CPT}
    &\ \sup_{y_0 \geq R_{\rm A}} \widetilde V_{\rm P} (y_0), \; \text{ where } \; \widetilde V_{\rm P} (y_0) = \sup_{(Z,\Gamma) \in \Vc} \E^{\P^{\nu^\star(\Gamma)}} \big[ U_{\rm P} (L_T) \big], \\ 
    \text{ with } &\ 
    \drm L_t = - \dfrac12 \big( |u^\star(\Gamma_t)|^{-2} + \gamma_{\rm A} |Z_t|^2 |u^\star(\Gamma_t)|^2 + 2h |u^\star(\Gamma_t)|^2 \big) \drm t + u^\star(\Gamma_t)(1-Z_t) \drm W_t, \; t \in [0,T], \nonumber
\end{align}
and with initial condition $L_0 = x_0 - U_{\rm A}^{-1}(y_0)$.
Note that for fixed $y_0 \geq R_{\rm A}$, the previous optimisation problem is a standard stochastic control problem with a one-dimensional state variable $L$ controlled through the parameters $(Z,\Gamma) \in \Vc$. For the sake of completeness, the solution is provided at the end of this section, see \Cref{lem:soution_ex_FB}. 

\medskip

If we follow our approach, we first need to fix the quadratic variation density process $\Sigma \in \Sc$. Note here that, since $\Sigma_t = |\nu_t|^2$, $t \in [0,T]$, and $\nu$ takes values in $(0,1]$, $\Sigma$ is also restricted to take values in $(0,1]$. The (constrained) Hamiltonian $\Hc_{\rm A}^\circ$ is defined by
\begin{align*}
    \Hc_{\rm A}^\circ (S) = \dfrac12 \sup_{u \in (0,1] : u^2=S} \big\{- u^{-2} \big\} = - \frac{1}{2 S}, \; \text{ for } \; u^\circ(S) := \sqrt{S}, \; S \in (0,1],
\end{align*}
and for $\Sigma \in \Sc$, the relevant form of contract is given by
\begin{align*}
    \xi^\circ = Y_T^{y_0,Z} = U_{\rm A}^{-1}(y_0) - \int_0^T \Hc^\circ_{\rm A}(\Sigma_t) \drm t + \int_0^T Z_t \drm X_t + \dfrac12  \gamma_{\rm A} \int_0^T |Z_t|^2 \drm \langle X \rangle_t,
\end{align*}
for $y_0 \in \R$ and $Z \in \Vc^\circ(\Sigma)$. For any $\Sigma \in \Sc$ and $\xi^\circ$ of the previous form, the agent's best response is given by the maximiser of the (constrained) Hamiltonian $\Hc_{\rm A}^\circ$, namely $\nu^\circ_t = u^\circ(\Sigma_t) = \sqrt{\Sigma_t}$, $t \in [0,T]$, which allows to achieve the best value $y_0 \geq R_{\rm A}$. As before, we can compute the dynamics of $X$ and $Y^{y_0,Z}$ under the optimal effort to obtain
\begin{align*}
    \drm X_t &= u^\circ(\Sigma_t) \drm W_t, \text{ and } \;
    \drm Y_t^{y_0,Z} = \dfrac12 \big( |u^\circ(\Sigma_t)|^{-2} +  \gamma_{\rm A} |Z_t|^2 |u^\circ(\Sigma_t)|^{2} \big)\drm t + Z_t u^\circ(\Sigma_t) \drm W_t, \quad t \in [0,T].
\end{align*}
Defining $L^\circ_t := X_t - Y_t^{y_0,Z} - h \langle X \rangle_t$, $t \in [0,T]$, the principal's problem becomes a standard stochastic control problem with a one-dimensional state variable $L^\circ$, controlled through $\Sigma \in \Sc$ and $Z \in \Vc^\circ(\Sigma)$:
\begin{align}\label{ex:principal_FB}
    &\ \sup_{y_0 \geq R_{\rm A}} \widetilde V_{\rm P} (y_0), \; \text{ where } \; \widetilde V_{\rm P} (y_0) = \sup_{\Sigma \in \Sc} \sup_{Z \in \Vc^\circ(\Sigma)} \E^{\P^{\nu^\circ(\Sigma)}} \big[ U_{\rm P} (L^\circ_T) \big], \\ 
    \text{ with } &\ 
    \drm L^\circ_t = - \dfrac12 \big( |u^\circ(\Sigma_t)|^{-2} + \gamma_{\rm A} |Z_t|^2 |u^\circ(\Sigma_t)|^2 + 2h \big| u^\circ(\Sigma_t) \big|^2\big) \drm t + u^\circ(\Sigma_t)(1-Z_t) \drm W_t, \; t \in [0,T], \nonumber
\end{align}
with initial condition $L_0 = x_0 - U_{\rm A}^{-1}(y_0)$.

\medskip

Comparing \eqref{ex:principal_CPT} and \eqref{ex:principal_FB}, we clearly need to have equality between the optimal controls to ensure that both problems lead to the same results. This necessitates $u^\circ(S) = u^\star(\gamma)$, for $(\gamma,S) \in \R \times (0,1]$, and we therefore intuit the following correspondence:
\begin{align*}
    S^\star(\gamma) := 
    \begin{cases}
        |\gamma|^{-1/2} & \text{if } \gamma < -1, \\
        1 & \text{if } \gamma \geq -1.
    \end{cases}
\end{align*}
This is in line with the correspondence given by \Cref{lem:pb2<=pb4} since, for $\gamma \in \R$, we clearly have
\begin{align*}
    \argmax_{S \in (0,1]} \bigg\{ \Hc^\circ_{\rm A} (S) + \dfrac12 \gamma S \bigg\}
    = \argmax_{S \in (0,1]} \big\{ - S^{-1} + \gamma S  \big\}
    = \begin{cases}
        |\gamma|^{-1/2} & \text{if } \gamma < -1, \\
        1 & \text{if } \gamma \geq -1.
    \end{cases}
\end{align*}
On the other hand, for $S \in (0,1]$ and $\gamma \in \R$, we compute
\begin{align*}
    \Hc_{\rm A} (\gamma) - \dfrac12 \gamma S
    = \dfrac12 \sup_{u \in (0,1]} \big\{ \gamma u^2 - u^{-2} \big\} - \dfrac12 \gamma S 
    = \begin{cases}
        - |\gamma|^{1/2} - \dfrac12 \gamma S & \text{if } \gamma < -1, \\[0.5em]
        \dfrac12 \big( \gamma - 1 - \gamma S \big) & \text{if } \gamma \geq -1.
    \end{cases}
\end{align*}
\modif{Since $S \in (0,1]$, we first show that \Cref{ass:duality} is satisfied. Indeed,
\begin{align*}
    \min_{\gamma \in \R}\left\{\Hc_{\rm A}(\gamma)-\dfrac12 \gamma S\right\}=-\frac{1}{2S} = \Hc^\circ_{\rm A}(S), \quad S\in(0,1],
\end{align*}
with the minimum achieved for $\gamma^\star(S) = S^{-2}$. This is inline with \Cref{lem:pb2>=pb4}.}

\begin{remark}\label{lem:soution_ex_FB}
To solve \eqref{ex:principal_CPT} or \eqref{ex:principal_FB}, one can compute the principal's Hamiltonian and write the corresponding {\rm HJB} equation. Recall that both problems are equivalent, but here, \eqref{ex:principal_FB} may be more straightforward to solve compared to \eqref{ex:principal_CPT}, where one would need to distinguish cases depending on the value of $\Gamma$. Through standard techniques, one obtains the following values and efforts for the agent and the principal, respectively, 
\begin{align*}
    V_{\rm A}(\xi) &= R_{\rm A}, \; \text{ for } \; 
    \nu^\star \equiv 1 \wedge \big( 2h + \gamma_{\rm A} \overline{\gamma} \big)^{-1/4}, \; \text{ with } \; \overline{\gamma} := \dfrac{\gamma_{\rm P}}{\gamma_{\rm P} + \gamma_{\rm A}}; \\ \text{ and } \; 
    V_{\rm P} &= - \exp \Big(- \gamma_{\rm P} \Big(x_0 - U_{\rm A}^{-1}(R_{\rm A}) - b(0) \Big) \Big), \; \text{ for } \; 
    Z^\star \equiv \overline{\gamma}, \; 
    \Sigma^\star \equiv 1 \wedge \big( 2h + \gamma_{\rm A} \overline{\gamma} \big)^{-1/2}.
\end{align*}
One can verify that solving \eqref{ex:principal_CPT} lead to the exact same results, but with the optimal sensitivity parameter
\begin{align*}
    \Gamma^\star \equiv - \big| \Sigma^\star \big|^{-2} \equiv - \big( 1 \vee \big(2h + \gamma_{\rm A} \overline{\gamma} \big) \big).
\end{align*}
\end{remark}

\subsection{A more intricate example}\label{ss:example2}

We now consider a model inspired by the demand-response problem introduced and solved by \citeayn{aid2022optimal}. In this model\footnote{We will only briefly describe the model here and refer to \cite{aid2022optimal} for further discussion and motivation.}, the agent can impact, through an admissible control $\nu:=(\alpha,\beta) \in \Uc$, both the mean and the volatility of his electricity consumption for various specific usages, such as heating, air conditioning, etc. His deviation with respect to his original consumption $x_0 \in \R_+$ is denoted $X$, and satisfies:
\begin{align*}
    X_t = x_0 - \int_0^t \alpha_s \cdot \mathbf{1}_d \drm s + \int_0^t \sigma (\beta_s) \drm W_s, \quad t\in[0,T],
\end{align*}
where $W$ is a $d$-dimensional Brownian Motion.
The volatility function $\sigma$ in the previous dynamics is defined by
\begin{align*}
    \sigma(b):=(\sigma_1\sqrt{b_1},\dots,\sigma_d\sqrt{b_d}), \; b := (b_1,\dots,b_d) \in [0,1]^d, \; \text{ for } \; \sigma_k>0, \; k \in \{1, \dots, d\}.
\end{align*}
As usual, it is assumed that the principal observes the output process $X$ in continuous time, and therefore its quadratic variation. However, contrary to the previous example, these observations are not sufficient to deduce the agent's effective effort, notably because of the dimension of the underlying noise compare to the observed processes.

\medskip

Given an admissible contract $\xi$, the agent's expected utility to be maximised is defined as follows,
\begin{align*}
    J_{\rm A} (\xi, \nu) := \E^{\P^\nu} \bigg[ U_{\rm A} \bigg( \xi + \int_0^T \big( \kappa X_t - c(\nu_t) \big) \drm t \bigg) \bigg], \; \text{ with } \; U_{\rm A} (x) = - \erm^{- \gamma_{\rm A} x}, \; x \in \R, \; \gamma_{\rm A} > 0, \; \kappa > 0,
\end{align*}
where the cost of effort is given, for any $u := (a,b) \in A \times B$ with $A := \R_+^d$ and $B := \R_+^d$, by
\begin{align*}
    c(u) := c_1(a) + \dfrac12 c_2(b), \; \text{ with } \; 
    c_1(a) := \dfrac12 \sum_{k=1}^d \frac{a_k^2}{\mu_k} \; \text{ and } \; 
    c_2(b) := \sum_{k=1}^d \dfrac{\sigma_k^2}{\lambda_k b_k},
\end{align*}
for some parameters $\lambda_k>0$ and $\mu_k>0$, $k \in \{1,\dots,d\}$.\footnote{Note that here, for computational convenience and illustration purposes, we slightly simplify the model, in particular the sets $A$ and $B$. Indeed, in the original model, these sets are bounded, which will complicate the agent's problem in our approach.}
The principal, who represents an electricity producer/supplier, wishes to incentivise the consumer to decrease the mean and the volatility of his consumption during the contracting period, in order to maximise his expected utility, defined as follows:
\begin{align*}
    J_{\rm P} (\xi,\nu) := \E^{\P^\nu} \bigg[ U_{\rm P} \bigg(- \xi - \theta \int_0^T X_t \drm t - \dfrac12 h \langle X \rangle_T \bigg) \bigg], \; \text{ with } \; U_{\rm P} (x) = - \erm^{- \gamma_{\rm P} x}, \; x \in \R, \; \gamma_{\rm P} > 0, \; h > 0, \; \theta > 0.
\end{align*}

\medskip

If we follow the approach developed in \cite{cvitanic2018dynamic}, and used in \cite{aid2022optimal}, the optimal form of contract is given by
\begin{align*}
    \xi = Y_T^{y_0,Z,\Gamma} = U_{\rm A}^{-1}(y_0) - \int_0^T \Hc_{\rm A}(X_t, Z_t, \Gamma_t) \drm t + \int_0^T Z_t \drm X_t + \dfrac12 \int_0^T \big(\Gamma_t + \gamma_{\rm A} |Z_t|^2 \big) \drm \langle X \rangle_t,
\end{align*}
for $y_0 \in \R$, a pair of $\R$-valued processes $(Z,\Gamma)$, and where the agent's Hamiltonian is defined for $(x,z,\gamma) \in \R^3$ by
\begin{align*}
    \Hc_{\rm A}(x,z,\gamma) &= \sup_{(a,b)\in A\times B} \bigg\{ -a \cdot \mathbf{1}_d z + \dfrac12 \gamma \sigma(b) \sigma^\top(b) + \kappa x-c_1(a) - \dfrac12 c_2(b) \bigg\}\\
    &= \kappa x - \inf_{a\in A} \big\{a \cdot \mathbf{1}_d z + c_1(a) \big\} - \dfrac12 \inf_{b\in B}\big\{ c_2(b)-\gamma \sigma(b) \sigma^\top(b) \big\}.
\end{align*}
As usual, computing the optimiser of the Hamiltonian will determine the agent's best response functions, \textit{i.e.}
\begin{align*}
    a^\star_k(z) = - \mu_k z \vee 0, \; \text{ and } \;
    b^\star_k(\gamma) = (-\lambda_k \gamma)^{-1/2}, \; \gamma < 0, \quad k=1,\dots,d.
\end{align*}
We can then reformulate the principal's problem as follows,
\begin{align}\label{ex2:principal_CPT}
    &\ \sup_{y_0 \geq R_{\rm A}} \widetilde V_{\rm P} (y_0), \; \text{ where } \; \widetilde V_{\rm P} (y_0) = \sup_{(Z,\Gamma) \in \Vc} \E^{\P^{\nu^\star(Z,\Gamma)}} \bigg[ U_{\rm P} \bigg(- Y^{y_0,Z,\Gamma}_T - \theta \int_0^T X_t \drm t - \dfrac12 h \langle X \rangle_T \bigg) \bigg],
\end{align}
with dynamics of state variables $X$ and $Y^{y_0,Z,\Gamma}$ under the optimal efforts given by
\begin{align*}
    \drm X_t &= - \widebar \mu (Z_t \wedge 0) \drm t + (-\Gamma_t)^{-1/4} \sum_{k=1}^d \dfrac{\sigma_k}{\lambda_k^{1/4}} \drm W_t^k, \\
    \drm Y^{y_0,Z,\Gamma}_t &= \dfrac12 \bigg( -2 \kappa X_t 
    + \widebar \mu (Z_t \wedge 0)^2
    + \widebar \sigma \bigg( \sqrt{-\Gamma_t} + \gamma_{\rm A} \dfrac{|Z_t|^2}{\sqrt{-\Gamma_t}} \bigg) \bigg) \drm t
    + \dfrac{Z_t}{(-\Gamma_t)^{1/4}} \sum_{k=1}^d \dfrac{\sigma_k}{\lambda_k^{1/4}} \drm W_t^k, 
\end{align*}
for $t \in [0,T]$, with initial conditions $X_0 = x_0$ and $Y^{y_0,Z,\Gamma}_0 = U_{\rm A}^{-1}(y_0)$, and where we denoted
\begin{align*}
    \widebar \mu :=\sum_{k=1}^d \mu_k, \; \text { and } \; 
    \widebar \sigma := \sum_{k=1}^d \dfrac{\sigma_k^2}{\sqrt{\lambda_k}}.
\end{align*}

Alternatively, if we proceed with our approach, the constraint Hamiltonian is given for $(x,z,S) \in \R^2 \times \R^\star_+$ by
\begin{align*}
    \Hc^\circ_{\rm A}(x,z,S) 
    &= \kappa x - \inf_{a\in \R_+^d} \big\{a \cdot \mathbf{1}_d z + c_1(a) \big\} - \dfrac12 \inf_{b\in \R_+^d \text{ s.t }\sigma(b)\sigma(b)^\top=S}\big\{ c_2(b) \big\},
\end{align*}
and leads to the following optimiser
\begin{align*}
    a^\circ_k(z) = - \mu_k (z \wedge 0) \; \text{ and }\;
    b^\circ_k(S) = \frac{S}{\widebar \sigma \sqrt{\lambda_k}}.
\end{align*}
We can then compute the dynamics of $X$ and $Y^{y_0,Z}$ under the optimal effort, and rewrite the principal's problem as a more standard control problem, similar to \eqref{ex2:principal_CPT}.
It is nevertheless already obvious that, to have equivalence between the two problems, we need to ensure that the quadratic variation of $X$ in both problems coincides, leading to the following relationship between $S \in \R_+^\star$ and $\gamma \in \R_-$:
\begin{align*}
    \sum_{k = 1}^d \sigma_k^2 b_k^\circ(\gamma)  = \sum_{k = 1}^d \sigma_k^2 b_k^\star(\gamma) 
    \quad \Leftrightarrow \quad 
    S = \dfrac{\widebar \sigma}{\sqrt{-\gamma}}
    \quad \Leftrightarrow \quad 
    \gamma = - \dfrac{\widebar \sigma^2}{S^2}.
\end{align*}
One can verify that this coincides with the correspondence given by \Cref{lem:pb2<=pb4,lem:pb2>=pb4}. Indeed, we have
\begin{align*}
    \modif{S^\star(\gamma):=} \argmax_{S>0} \bigg\{ \Hc_{\rm A}^\circ(x,z,S)+\dfrac12 \gamma S \bigg\} 
    &= \argmax_{S>0} \bigg\{- \frac{\widebar \sigma^2}{S} + \gamma S \bigg\} 
    = \dfrac{\widebar \sigma}{\sqrt{- \gamma}}, \quad \gamma < 0, \\
    \text{ and } \; 
    \modif{\gamma^\star(S):=} \argmin_{\gamma<0} \bigg\{ \Hc_{\rm A}(x,z,\gamma)-\dfrac12 \gamma S \bigg\} 
    &= \argmin_{\gamma<0} \bigg\{  -2 \widebar \sigma (- \gamma)^{1/2} - \gamma S \bigg\} 
    = - \dfrac{\widebar \sigma^2}{S^2}, \quad S > 0.
\end{align*}
\modif{In particular, \Cref{ass:duality} is satisfied since 
\begin{align*}
    \min_{\gamma<0} \bigg\{ \Hc_{\rm A}(x,z,\gamma)-\dfrac12 \gamma S \bigg\} = \kappa x + \dfrac12 \widebar \mu (z \wedge0)^2 - \dfrac{\widebar \sigma^2}{S} = \Hc_{\rm A}^\circ(x,z,S), \quad S > 0,
\end{align*}
and the minimum is achieved at $\gamma^\star(S)$.
}

\modif{
\subsection{A `counter-example'}\label{ss:counter-example}

In this example, we highlight that \Cref{ass:duality} may not always be satisfied, and that, in this case, there may be a difference between:
\begin{itemize}
    \item the value of the `first-best' formulation, namely $V^\circ_{\rm P}$ in \Cref{pb:PA_FB_reformulation}, or equivalently $\widetilde V^\circ_{\rm P}$ in \Cref{pb:first-best_solution},
    \item the value of the original problem when restricting to contracts of the form \eqref{eq:contract_CPT}, \textit{i.e.} $\widetilde V_{\rm P}$ in \Cref{pb:PA_solution_CPT}.
\end{itemize} 
For this, we consider the controlled state equation \eqref{eq:X_dynamic_example} driven by a one-dimensional Brownian Motion $W$ as in the first example in \Cref{ss:example1}, but where now the control process $\nu$ can take values in $U := [-1,1]$. We note that achievable quadratic variations take values in $[0,1]$. We assume further that the agent's cost function is given by $c(u) = 1-u^4$ for $u \in [-1,1]$, thus non-negative and bounded as required.

\medskip

We first highlight that \Cref{ass:duality} is not satisfied in this framework. Indeed, let $S \in [0,1]$, we have
\begin{align}\label{eq:max_constraint}
	\Hc_{\rm A}^\circ(S) := \sup_{u \in U : u^2 = S} \big\{ - c(u) \big\}
	= \max \Big\{ \big(\sqrt{S} \big)^4 - 1, \big(-\sqrt{S} \big)^4 - 1 \Big\}
	= S^2 -1, \quad \text{ for } \; u^\circ(S) = \pm \sqrt{S}.
\end{align}
The `constraint' Hamiltonian $\Hc_{\rm A}^\circ$ is proper, continuous, but convex instead of concave, implying that \Cref{ass:duality} is not satisfied for all $S \in [0,1]$. Indeed, for any $\gamma \in \R$, we have
\begin{align}\label{eq:max_hamiltonian}
	\Hc_{\rm A}(\gamma) &:= \sup_{u \in U} \bigg\{\dfrac12 \gamma u^2 - c(u) \bigg\}
	= \sup_{S \in [0,1]} \bigg\{ \sup_{u \in [-1,1] : u^2 = S} \bigg\{\dfrac12 \gamma u^2 + u^4 \bigg\} \bigg\} -1
	= \sup_{S \in [0,1]} \bigg\{ \dfrac12 \gamma S + S^2 \bigg\} -1 \nonumber \\
	&= \begin{cases}
		\gamma/2 &\mbox{ if } \gamma > -2, \text{ for } S^\star(\gamma) = 1, \text{ i.e. } u^\star(\gamma) \in \{-1,1\}; \\
		-1 &\mbox{ if } \gamma = -2, \text{ for } S^\star(\gamma) \in \{0,1\}, \text{ i.e. } u^\star(\gamma) \in \{-1,0,1\}; \\
		-1 &\mbox{ if } \gamma < -2, \text{ for } S^\star(\gamma) = 0, \text{ i.e. } u^\star(\gamma) =0.
	\end{cases}
\end{align}
Computing the convex conjugate of $\Hc_{\rm A}$, we obtain
\begin{align*}
	\inf_{\gamma \in \R} \bigg\{ \Hc_{\rm A}(\gamma) - \dfrac12 \gamma S \bigg\}
	= S-1, \quad \text{ for } \gamma^\star = -2.
\end{align*}
We thus have 
\begin{align*}
    \inf_{\gamma \in \R} \bigg\{ \Hc_{\rm A}(\gamma) - \dfrac12 \gamma S \bigg\} - \Hc_{\rm A}^\circ(S) = S - S^2,
\end{align*}
which is always non-negative, but (strictly) positive on $(0,1)$. In other words, \Cref{ass:duality} is not verified, unless we restrict the study to $S \in \{0,1\}$. Intuitively, if the principal's optimal choice for the quadratic variation is at either \( S = 0 \) or \( S = 1 \), then no issue should arise, since \Cref{ass:duality} is satisfied at these two points. However, if the optimal value for the quadratic variation lies in the interior of the interval \( (0,1) \), the `first-best' value should exceed the value obtained by restricting to contracts of the form~\eqref{eq:contract_CPT}. Indeed, as highlighted above when computing the maximisers of the Hamiltonian $\Hc_{\rm A}$, these contracts can only incentivise optimal efforts taking values in $\{-1,0,1\}$, thus leading to a quadratic variation taking values in $\{0,1\}$ only. We confirm this intuition below.

\medskip

We consider the following criterion for the principal, in which the cost function $c_{\rm P}$ will be chosen later,
\begin{align}\label{eq:principal}
	J_{\rm P} (\xi,\Sigma) := \E \bigg[ X_T - \xi - \int_0^T c_{\rm P} \big(\Sigma_t \big) \drm t \bigg].
\end{align}
In the `first-best' formulation, we know that the optimal contract form is given by
\begin{align*}
	\xi = y_0 - \int_0^T \Hc_{\rm A}^\circ (\Sigma_t) \drm t + \int_0^T Z_t \drm X_t,
\end{align*}
where the pair of processes $(Z,\Sigma)$ has to be chosen by the principal. The optimal efforts are given by the maximiser of the constrained Hamiltonian $\Hc^\circ_{\rm A}$, here $\nu_t^\circ = \pm \sqrt{\Sigma_t}$, which leads to $\Hc_{\rm A}^\circ(\Sigma_t) = \Sigma_t^2 -1$, $t \in [0,T]$, as shown in \eqref{eq:max_constraint}. Using the dynamics of $X$ and $\xi$ under these optimal efforts, we can compute
\begin{align*}
	J_{\rm P} (\xi,\Sigma) = x_0 - y_0 - T + \E \bigg[ \int_0^T \big( \Sigma_t^2 - c_{\rm P} (\Sigma_t) \big) \drm t \bigg].
\end{align*}
By pointwise optimisation, the optimal $\Sigma$ should be given by
\begin{align*}
	\Sigma_t^\star \in \argmax_{S \in [0,1]} \big\{ S^2 - c_{\rm P} (S) \big\}, \quad t \in [0,T].
\end{align*}
Now, by choosing the cost function $c_{\rm P} (S) = S^3$, we obtain a maximum at $4/27$ achieved for $S^\star = 2/3 \in (0,1)$. In other words, the optimal quadratic variation is neither $0$ or $1$, and one can thus expect a duality gap in this case. 

\medskip

Indeed, one can also compute the principal's value when she can only offer contracts of the form \eqref{eq:contract_CPT}. In this case, we still consider the criterion \eqref{eq:principal}, but with the difference that $\Sigma$ is not directly chosen by the principal. In particular, given a contract of the form \eqref{eq:contract_CPT} indexed through a parameter $\Gamma$, we have $\Sigma_t = |u^\star(\Gamma_t)|^2$, $t \in [0,T]$, where $u^\star$ is defined through \eqref{eq:max_hamiltonian} as the maximiser of the Hamiltonian $\Hc_{\rm A}$. Using this contract form and the associated optimal efforts, we can then compute
\begin{align*}
	J_{\rm P} \big(\xi,| u^\star(\Gamma_\cdot)|^2 \big) = x_0 - y_0 + \E \bigg[ \int_0^T \bigg( \Hc_{\rm A}(\Gamma_t) - \dfrac12 \Gamma_t | u^\star(\Gamma_t)|^2  - c_{\rm P} \big( | u^\star(\Gamma_t)|^2 \big) \bigg)\drm t \bigg],
\end{align*}
where now the process $\Gamma$ has to be chosen optimally. Again by pointwise optimisation, we should have
\begin{align*}
	\Gamma_t^\star \in \argmax_{\gamma \in \R} \bigg\{ \Hc_{\rm A}(\gamma) - \dfrac12 \gamma | u^\star(\gamma) |^2  - c_{\rm P} \big( | u^\star(\gamma) |^2 \big) \bigg\}, \quad t \in [0,T].
\end{align*}
By choosing $c_{\rm P}(S) = S^3$ as before, the maximisation problem becomes
\begin{align*}
	\max_{\gamma \in \R} \bigg\{ \Hc_{\rm A}(\gamma) - \dfrac12 \gamma | u^\star(\gamma) |^2  - | u^\star(\gamma)|^6  \bigg\} = -1, \quad \text{for } \; \gamma^\star = -2.
\end{align*}
Therefore, while the value in the `first-best' case is given by $x_0 - y_0 - 23T/27$, the maximum possible value achievable with the contracts \eqref{eq:contract_CPT} in the original `second-best' problem is lower, equal to $x_0 -y_0 - T$.

\medskip

This illustrative example highlights that if \Cref{ass:duality} does not hold, the contract form \eqref{eq:contract_CPT} may fail to achieve the `first-best' value. As mentioned in \Cref{rk:duality}, without \Cref{ass:duality}, the proof of optimality of the contract form \eqref{eq:contract_CPT} via 2BSDEs may also break down. Therefore, while this example illustrates a difference between the value of the original problem when restricting attention to contracts of the form \eqref{eq:contract_CPT} (namely, \Cref{pb:PA_solution_CPT}) and the value of the `first-best' reformulation (\Cref{pb:PA_FB_reformulation,pb:first-best_solution}), it remains unclear whether this reflects an intrinsic `gap' between \Cref{pb:PA_original} and \Cref{pb:PA_FB_reformulation}, or whether it is possible to design a new class of \emph{penalisation} contracts to bridge this gap. We believe this is a very interesting question, which we leave for further research.}


\section{Conclusion}\label{sec:conclusion}

Principal--agent problems with volatility control are notoriously challenging to study. Specifically, the only existing approach so far, introduced by \citeayn{cvitanic2018dynamic}, relies on the sophisticated theory of \emph{second-order} BSDEs. More precisely, the authors establish the equality between the principal's value $V_{\rm P}$ in the original problem (\Cref{pb:PA_original}) and the value $\widetilde V_{\rm P}$ of a more standard stochastic control problem (\Cref{pb:PA_solution_CPT}). While it is clear from the restriction to contracts of the form \eqref{eq:X-dynamics-CPT} that $V_{\rm P} \geq \widetilde V_{\rm P}$, the converse inequality requires to prove that such restriction is without loss of generality, which necessitates the use of \emph{second-order} BSDEs. Even though this approach indeed allows a thorough analysis of the principal--agent problem with volatility control, and therefore of the applications arising from this framework, the mathematical complexity of this approach hampers the study of natural and relevant extensions, especially to multi-agent frameworks. Moreover, this mathematically challenging approach also prevents a wider dissemination of these models and their relevant applications in many fields, from operation research to economics.

\medskip

In contrast, the novel approach we develop only relies on the theory of BSDEs, thanks to the introduction of a `first-best' counterpart of the original problem, in which the principal directly controls the density of the output's quadratic variation  (\Cref{pb:PA_FB_reformulation}). This alternative problem is purely theoretical: it can be seen as a dual to \Cref{pb:PA_original}, and is not necessarily meaningful in terms of application. Nevertheless, compared to \Cref{pb:PA_original}, \Cref{pb:PA_FB_reformulation} can be solved in a more straightforward way, following the so-called Sannikov's trick. More precisely, one can prove using BSDEs that the restriction to contracts of the form \eqref{eq:contract_FB} is without loss of generality, which implies that the principal's problem in this `first-best' reformulation is actually equivalent to another standard stochastic control problem, namely $\widetilde V_{\rm P}^\circ$ defined in \Cref{pb:first-best_solution}.
It then remains to show that the value $V_{\rm P}^\circ$ of the `first-best' reformulated problem, which naturally satisfies $V_{\rm P}^\circ \geq V_{\rm P}$, can be achieved. For this, it suffices to use the contract form \eqref{eq:contract_CPT} considered in \cite{cvitanic2018dynamic}, and thus show that \Cref{pb:PA_solution_CPT} and \Cref{pb:first-best_solution} are actually equivalent. As a result, instead of choosing the sensitivity parameter $\Gamma$ in the contract form \eqref{eq:contract_CPT}, one can directly optimise the density process $\Sigma$ of the output's quadratic variation. Conversely, given a density process $\Sigma$, one can recover from the correspondence in \Cref{lem:pb2>=pb4} the appropriate pay-for-performance sensitivity $\Gamma$ in the contract form \eqref{eq:contract_CPT}. We finally obtain that the problems introduced throughout this paper are all equivalent \modif{under our standing \Cref{ass:weak_uniqueness,ass:duality}}, which is graphically summarised by \Cref{fig:summary}.

\begin{figure}[ht!]
\begin{center}
\begin{tikzpicture}
\tikzstyle{pb1}=[rectangle,draw,rounded corners=4pt,fill=red!30]
\tikzstyle{pb2}=[rectangle,draw,rounded corners=4pt, fill=orange!30]
\tikzstyle{pb3}=[rectangle,draw,rounded corners=4pt, fill=blue!30]
\tikzstyle{pb4}=[rectangle,draw,rounded corners=4pt, fill=green!30]

\node[pb1] (pb1) at (-3,1.5) {\Cref{pb:PA_original}};
\node[pb2] (pb2) at (-3,-1.5) {\Cref{pb:PA_solution_CPT}};
\node[pb3] (pb3) at (3,1.5) {\Cref{pb:PA_FB_reformulation}};
\node[pb4] (pb4) at (3,-1.5) {\Cref{pb:first-best_solution}};

\draw[->,thick] (pb1) edge[bend left] node[right,fill=white]{\Cref{lem:pb1>pb2}, or \cite{cvitanic2018dynamic}} (pb2);
\draw[<-,thick] (pb1) edge[bend left] node[midway,fill=white]{\Cref{lem:pb3>pb1}} (pb3);
\draw[<->,ultra thick,green] (pb3) -- (pb4) node[right,midway,fill=white]{\Cref{thm:solution_FB}};
\draw[->, dashed] (pb2) edge[bend left] node[left,fill=white]{By 2BSDEs in \cite{cvitanic2018dynamic}} (pb1);
\draw[<->, ultra thick,green] (pb1) -- (pb3) node[midway,fill=white]{\Cref{thm:main}};
\draw[<->,thick] (pb2) -- (pb4) node[midway,fill=white]{\Cref{prop:pb2=pb4}};
\end{tikzpicture}
\caption{Equivalence of the problems}
\label{fig:summary}
\end{center}
\end{figure}
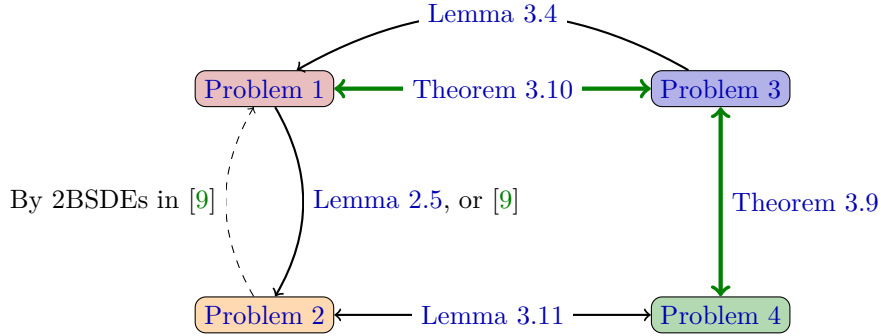

We nevertheless remark that, at the end, and depending on the framework under consideration, \Cref{pb:PA_solution_CPT} may still be easier to solve than \Cref{pb:first-best_solution}. Indeed, as illustrated through \Cref{sec:examples}, the agent's problem \Cref{pb:first-best_solution} involves a constrained optimisation, making it potentially more complex to solve in practice. However, this does not undermine the effectiveness of our approach, since its main goal is to provide a theoretically simpler method to establish the optimality of the relevant contract form. Ultimately, since all problems are equivalent, this approach offers the flexibility to choose the most practical problem to solve. 
We also emphasise that the contract form \eqref{eq:contract_CPT}, originally introduced in \cite{cvitanic2018dynamic}, is essential for proving the equivalence between the original problem and its `first-best' reformulation. Nevertheless, the key message of our approach is that we do not need to rely on 2BSDEs to demonstrate the optimality of these contracts; their optimality naturally comes from the fact that they allow the principal to achieve the best utility possible. 

\medskip

We believe that this alternative, simpler approach will provide a more straightforward pathway for extending principal--agent problems to more complex frameworks, involving for example multi-agent settings, mean-field interactions, or non-continuous output processes. For instance, solving principal--multi-agent problems with volatility control via the original approach in \cite{cvitanic2018dynamic} requires the consideration of `multidimensional 2BSDEs', as highlighted in \cite{hubert2023continuoustime}, for which the theory is actually not yet developed. In contrast, following our proposed approach, one could rely on the more mature theory of multidimensional BSDEs to show the equivalence between the original problem and its `first-best' reformulation, thereby ensuring optimality of appropriate revealing contracts without relying on multidimensional 2BSDEs. Furthermore, this simplified approach shall facilitate practical and more widespread implementation of principal--agent problems with volatility control in diverse fields, beyond mathematical finance. 

\medskip

\modifreview{To conclude, we want to emphasise that while our approach provides a simpler and more accessible framework for solving principal–agent problems with volatility control, it relies on two key assumptions: \Cref{ass:weak_uniqueness}, which guarantees weak uniqueness and thus enables the use of the martingale representation theorem, and \Cref{ass:duality}, under which the contract form proposed in~\cite{cvitanic2018dynamic} remains optimal.
First, we view \Cref{ass:weak_uniqueness} as relatively mild and well motivated, particularly given the intended applications. Indeed, without weak uniqueness, the same choice of effort by the agent could generate different distributions for the controlled output process, making the model less meaningful from an economic standpoint. Moreover, although this assumption is central to our current BSDE-based arguments, it may be possible to relax it by considering \emph{generalised} BSDEs, \textit{i.e.} with an orthogonal martingale component. This approach would still remain simpler than the general 2BSDE framework, and could be extended for example to multi-agent settings, thanks to the existence of multi-dimensional versions of such BSDEs.}
\modif{Second, we have strong reason to believe that our main result, \Cref{thm:main}, will remain valid even without \Cref{ass:duality}. However, since the contract form from~\cite{cvitanic2018dynamic} may no longer be optimal in this case, this would require constructing a new class of contracts. While preliminary exploration in this direction is promising, a full rigorous treatment is left for future research.}

{\small \bibliography{bibliographyEmma.bib} }   

\end{document}